\documentclass[11pt,reqno]{article}
\usepackage{latexsym}
\usepackage{amssymb}
\usepackage{amsmath}
\usepackage[all]{xy}
\usepackage{color}
\usepackage{enumerate,url}

\newtheorem{thm}{Theorem}

\newtheorem{prop}{Proposition}

\newtheorem{cor}{Corollary}
\newtheorem{rem}{Remark}

\newcommand{\qed}{\hfill \fbox{}}

\newcommand{\beeq}{\begin{equation}}
\newcommand{\eeeqnar}{\end{eqnarray*}}
\newcommand{\beeqnar}{\begin{eqnarray*}}
\newcommand{\eeeq}{\end{equation}}
\newcommand{\beit}{\begin{itemize}}
\newcommand{\eeit}{\end{itemize}}
\newcommand{\bedes}{\begin{description}}
\newcommand{\eedes}{\end{description}}
\newcommand{\been}{\begin{enumerate}}
\newcommand{\eeen}{\end{enumerate}}

\def\ddu2{{\frac{\partial^2}{\partial u^2}}}
\def\Re{{\mathrm{Re}}}

\def\ZZ {{\mathbb Z}}

\def\CC {{\mathbb C}}
\def\PP {{\mathbb P}}

\newenvironment{proof}{\medskip
\noindent{\bf Proof: }}{{\hfill$\square$}{\medskip}}

\begin{document}

\title{Evaluating the Mahler measure of linear forms via the Kronecker limit
formula on complex projective space}
\author{James Cogdell \and Jay Jorgenson
\footnote{The second named author acknowledges grant support from several PSC-CUNY Awards, which are jointly funded by the Professional Staff Congress and The City University of New York.}\and Lejla Smajlovi\'{c}}
\maketitle

\begin{abstract} \noindent
In Cogdell et al., \it LMS Lecture Notes Series \bf 459, \rm 393--427 (2020), \rm the authors proved an
analogue of Kronecker's limit formula associated to any divisor $\mathcal D$ which is smooth in codimension one
on any smooth K\"ahler manifold $X$.  In the present article, we apply the aforementioned Kronecker limit
formula in the case when $X$ is complex projective space $\CC\PP^n$ for $n \geq 2$ and $\mathcal D$ is a
hyperplane, meaning the divisor of a linear form $P_D({z})$
for ${z} = (\mathcal{Z}_{j}) \in \CC\PP^n$.  Our main result is an explicit
evaluation of the Mahler measure of $P_{D}$ as a convergent series whose each term is given
in terms of rational numbers, multinomial coefficients, and the $L^{2}$-norm of the vector
of coefficients of $P_{D}$.
\end{abstract}

\section{Introduction}

\subsection{Mahler measure}
Let $P(x_{1},\cdots, x_{n}) \in \CC[x_{1},\cdots, x_{n}]$ be a polynomial in $n$ variables with
complex coefficients;  we assume that $P$ is not identically equal to zero.
The Mahler measure $M(P)$ of $P$ is defined by the expression
\begin{equation}\label{Mahler_definition}
M(P)=\exp \left({\frac  {1}{(2\pi )^{n}}}\int\limits_{0}^{{2\pi }}\int\limits_{0}^{{2\pi }}\cdots \int\limits_{0}^{{2\pi }}\log {\Bigl (}{\bigl |}
P(e^{{i\theta _{1}}},e^{{i\theta _{2}}},\ldots ,e^{{i\theta _{n}}}){\bigr |}{\Bigr )}\,d\theta _{1}\,d\theta _{2}\cdots d\theta _{n}\right).
\end{equation}
If $n=1$, we can write $P(x) = a_{d}x^{d} + \cdots + a_{1} x + a_{0} = a_{d} \prod\limits_{k=1}^{d}(x - \alpha_{k})$, in which case Jensen's formula implies that
\begin{equation}\label{Jensen_formula}
M(P) = \vert a_{d} \vert \prod\limits_{\vert \alpha_{j}\vert > 1} \vert \alpha_{j}\vert.
\end{equation}
As usual, one sets $m(P) = \log M(P)$ to denote the logarithmic Mahler measure of $P$.

Amongst the many articles involving Mahler measures, we shall highlight a few which we find particularly motivating.
In \cite{Sm08} the author presents an excellent survey of the many ways in which the Mahler measure of polynomials in
one variables are related to various questions in mathematics, including problems in algebraic number theory, ergodic theory,
knot theory, transcendental number theory and diophantine approximation, just to name a few.  In \cite{De97} the author
established a fascinating connection between Mahler measures and Deligne periods associated to mixed motives; see \cite{De09}
and \cite{De12} for subsequent development of the insight from \cite{De97}.  In \cite{Bo98}, the author undertakes a study
of numerical methods by which one can estimate Mahler measures and, as a result, is able to investigate some of the ideas
from \cite{De97}.  Since then, many authors have extend the study of Mahler measures both in the numerical direction as in
\cite{Bo98} and in the theoretical framework as in \cite{De97}.

On page 22 of \cite{BG06} the authors stated the definition of Mahler measure \eqref{Mahler_definition}
in the context of heights of polynomials, though subsequent discussion only considers the setting of one variable polynomials.
In \cite{Ma00} the author computed the arithmetic height, in the sense of Arakelov theory, of divisors in projective space.
Specifically, it was shown that a hypersurface defined over $\ZZ$ has canonical height which was expressed by the Mahler
measure of a defining polynomials; see page 107 of \cite{Ma00}.

The following observation is an underlying aspect of a considerable part of the aforementioned work:  In many instances,
Mahler measures can be expressed as special number, such as norms of algebraic numbers, arithmetic heights or special values of $L$-functions.
As such, the study of Mahler measures is intrinsically interesting.

\subsection{Mahler measure of a linear polynomial}

For this article, we will consider the specific setting of linear polynomials,
which itself has been the focus of attention; see, for example, \cite{R-VTV04} or \cite{Sm81}.
Let
\begin{equation} \label{defn:lin. polyn}
P_D(\mathcal{Z}_0,\ldots,\mathcal{Z}_n)=\mathcal{W}_0\mathcal{Z}_{0}+\mathcal{W}_1\mathcal{Z}_1+...+\mathcal{W}_n\mathcal{Z}_n
\end{equation}
denote the linear polynomial in $n$ complex projective coordinates variables, and
assume for now that $n \geq 2$.  We will parameterize the polynomial $P_D$ through the $(n+1)$-tuple
$$
D=(\mathcal{W}_0, \mathcal{W}_1,\ldots,\mathcal{W}_n)
$$
of its coefficients.  Of course, we assume that some $\mathcal{W}_{j}$ is not zero, and we set
$$
\Vert D \Vert^{2} = |\mathcal{W}_{0}|^{2} + \cdots +|\mathcal{W}_{n}|^{2}.
$$

Assuming that $\mathcal{W}_0 \neq 0$, one has that after dehomogenization the (logarithmic) Mahler measure $m(P_D)$ can be evaluated as
\begin{equation}\label{Mahler_definition1}
m(P_D)={\frac  {1}{(2\pi )^{n}}}\int\limits_{0}^{{2\pi }}\int\limits_{0}^{{2\pi }}\cdots \int\limits_{0}^{{2\pi }}\log {\Bigl (}{\bigl |}
P_D(1,e^{{i\theta _{1}}},e^{{i\theta _{2}}},\ldots ,e^{{i\theta _{n}}}){\bigr |}{\Bigr )}\,d\theta _{1}\,d\theta _{2}\cdots d\theta _{n}.
\end{equation}
In  \cite{R-VTV04}, the authors derived the bounds
\begin{equation}\label{degree_one_bounds}
\log||D||-\frac{1}{2}\gamma-2\leq m(P_D)\leq\log\Vert D\Vert,
\end{equation}
where $\gamma$ denotes Euler's constant. The upper bound in (\ref{degree_one_bounds}) is trivial; however, the lower bound follows from
reasonably extensive computations stemming from an infinite series expansion of the Mahler measure
in terms of certain weighted integrals of $J$-Bessel functions.  Indeed, one of the points made in \cite{R-VTV04} is
that their results are amenable to numerical estimation of $m(P_{D})$ for any linear polynomial $P_{D}$;
see, in particular, Corollary 1.4 on page 476 of \cite{R-VTV04}.

In \cite{Sm81} it is shown that for certain classes of linear polynomials in $n+1$ variables one can evaluate the
corresponding Mahler measure.  Some of the main results of \cite{Sm81} follow from clever applications of
Jensen's formula, thus the resulting formulas are similar to \eqref{Jensen_formula}.

To specialize further, let us now assume that for each $j$ we have that $\mathcal{W}_{j}=1$.  In \cite{BSWZ12} it
is shown that $m(P_D)= \frac{d}{ds}W_{n+1}(s) \Big|_{s=0}$ where
$$
W_{n+1}(s)=\int\limits_{0}^1 \cdots \int\limits_{0}^1 \left| \sum_{k=1}^{n+1}e^{2\pi i t_k}\right|^s dt_1\cdots dt_{n+1}.
$$
When studying the Mahler measure of $P_{D}$ the authors of \cite{BSWZ12} employed arguments from probability theory to $W_{n+1}(s)$, which they viewed as the $s$-th moment of an $(n+1)-$step random walk.  In doing so,
it is asserted on page 982 of \cite{BSWZ12} that
\begin{equation}\label{all_d_equal_one}
m(P_{D}) =\log(n+1)-\sum_{j=1}^{\infty}\frac{1}{2j}\sum_{k=0}^j\binom{j}{k} \frac{(-1)^k}{(n+1)^{2k}} W_{n+1}(2k)
\,\,\,\,\,
\textrm{\rm when $D=(1,1,\cdots,1)$.}
\end{equation}
As it turns out, equation \eqref{all_d_equal_one} is a special case of our Theorem \ref{thm:main} as stated below.

Finally, let us note that in certain special instances the
values of the Mahler measure of a linear polynomial $P_D$, have been computed explicitly.  When
$n=2$ and $D=(1,1,1)$, it is proved in \cite{Sm81} that
$$
m(1+z_1+z_2)=\frac{3\sqrt{3}}{4\pi}L(2,\chi_3)
$$
where $L(s,\chi_3)$ is the Dirichlet $L-$function associated to the non-principal odd character $\chi_3$ modulo $3$.  If $n=3$,
then it also is proved in \cite{Sm81} that
$$
m(1+z_1+z_2+z_3) =\frac{7}{2\pi^2}\zeta(3)
$$
where $\zeta(s)$ denotes the Riemann zeta function.  There are many other examples
of explicit evaluations of Mahler measures, far too many to provide an exhaustive listing.
However, it should be noted that each new evaluation is in and of itself interesting and
aids in the understanding the importance of Mahler measures.

\subsection{Our main results}

The purpose of this article is to develop a different means to evaluate Mahler measures of linear polynomials in $n+1$ complex variables.
Our approach is based on the following observation. A holomorphic section of a power of the canonical bundle on $n$-dimensional complex projective
space $\CC\PP^n$ can be realized as a homogeneous polynomial in $n+1$ projective coordinates.  Therefore, the log-norm of such a polynomial, which appears in the definition of
the Mahler measure, can be expressed in terms of the log-norm of a holomorphic form on $\CC\PP^n$ which, from the results of \cite{CJS20}, are related
to an integral over its divisor of a ``truncated'' Green's function, or resolvent kernel, on $\CC\PP^n$ by way of its Kronecker limit formula.  The
spectral expansion of Green's function on $\CC\PP^n$ yields a representation of the log-norm of the polynomial in terms of a certain infinite series which we are able to
explicitly evaluate.

Our first main result is the following theorem.

\begin{thm}\label{thm:main} With the notation as above, let $c(D)^2=(n+1)||D||^2$ and set
$$
a(n,k,D) = \sum_{\ell_0+\ldots+\ell_n=k, \ell_m\geq 0} \binom{k}{\ell_0,\ell_1,\ldots,\ell_n}^2|\mathcal{W}_0|^{2\ell_0}\cdots|\mathcal{W}_n|^{2\ell_n}
$$
where
$$
\binom{k}{\ell_0,\ell_1,\ldots,\ell_n}=\frac{k!}{\ell_0!\ell_1!\cdots\ell_n!}.
$$
is the multinomial coefficient. Then for $n \geq 3$ the logarithmic Mahler measure $m(P_D)$ of the linear polynomial $P_D$
is given by
\begin{equation}\label{main expression}
m(P_D)= \log c(D) -\frac{1}{2}
\sum_{j=1}^{\infty}\frac{1}{j}\sum_{k=0}^j\binom{j}{k} \frac{(-1)^k a(n,k,D)}{c(D)^{2k}}.
\end{equation}
\end{thm}

The sum over $j$ on the right-hand side of \eqref{main expression} is absolutely convergent.
However, it is not possible to view the series as a double series in $j$ and $k$ and then interchange the order of summation.
In particular, the series diverges when viewed as a sum over $j$ and for a fixed $k$.

We also obtain the following expression for $m(P_{D})$.

\begin{thm}\label{thm:main2}
For any integer $\ell \geq 1$, let $H_{\ell} = 1 + \frac{1}{2} \cdots + \frac{1}{\ell}$ and set
$$
S_{D}(\ell) = \sum_{j=1}^{\infty}\frac{2j+\ell}{j(j+\ell)}\sum_{k=0}^j\binom{j+\ell+k-1}{k}\binom{j}{k} \frac{(-1)^k a(n,k,D)}{c(D)^{2k}}.
$$
Then for any $n\geq 3$ and any $D$, we have that
\begin{equation}\label{measure formula 3}
m(P_D)= \log c(D) - \frac{1}{2}H_{1} -\frac{1}{2}S_{D}(1).
\end{equation}
Further, for any $n\geq 3$ and $\ell \geq 2$ we have that
\begin{equation}\label{measure formula 2}
m(P_D)= \log c(D) - \frac{1}{2}H_{\ell} -\frac{1}{2}S_{D}(\ell)
\end{equation}
provided $D \neq r(1,1,\cdots, 1)$ for some $r\neq 0$.
\end{thm}

From \eqref{measure formula 3},  we will prove that for all $n\geq 3$ and all $D$ one has that
\begin{equation}\label{measure formula 4}
m(P_D)=\log c(D) -\frac{1}{2} S_{D}(0).
\end{equation}
In summary, we will prove that \eqref{measure formula 2}
holds in the following cases:
\begin{enumerate}
\item[(i)] All $n \geq 3$ and all $D$ when $\ell=0$;
\item[(ii)] All $n \geq 3$ and all $D$ when $\ell=1$;
\item[(iii)] All $n \geq 3$ and all $\ell\geq 2$ provided $D \neq r(1,1,\cdots, 1)$ for some $r\neq 0$.
\end{enumerate}
In our concluding comments to this paper, we will discuss the exceptional instance in case (iii) as
well as the general setting when $n=2$.

As stated, \cite{Sm81} obtains explicit evaluations of Mahler measures for certain linear polynomials.  Thus,
by combining our main theorem with the formulas from \cite{Sm81}, we obtain many intriguing identities.  Along this
line, let us point out the following ``amusing'' corollary which comes from comparing our results to
those from \cite{R-VTV04}.

\begin{cor} \label{cor. comparison}
For any non-zero vector $D=(\mathcal{W}_0,\mathcal{W}_1,\ldots,\mathcal{W}_n)\in\CC^{n+1}$ one has that
\begin{equation}\label{comparison equation}
\sum_{j=1}^{\infty}\frac{1}{j}\sum_{k=0}^j\binom{j}{k} \frac{(-1)^k a(n,k,D)}{||D||^{2k}}\left(\frac{1}{(n+1)^k} - \frac{1}{k!}\right)=\log(n+1)+\gamma,
\end{equation}
where $\gamma$ denotes the Euler constant.
\end{cor}

\noindent
It is intriguing that the right-hand side of \eqref{comparison equation} is independent of $D$. When $n=0$, equation \eqref{comparison equation}
still makes sense and will follow from equation (2.4) of \cite{R-VTV04} after one would show that \eqref{comparison equation}  converges.
As such, it is possible that \eqref{comparison equation} could be proved directly, at least for some "small" values of $n$.

A further discussion
of additional identities is given in the concluding section of this article, see e.g. identity \eqref{comb identities2}.

In our proof of Theorem \ref{thm:main} and Theorem \ref{thm:main2} we obtain precise rates of convergence
of the infinite series involved.  Specifically, we obtain the following estimates.

\begin{thm}\label{thm:series} With notation as above, assume $n \geq 3$ and choose any $D\neq 0$.
Then there is an explicitly computable constant $G(n,D)$, which depends solely on $n$ and $D$, such that
for any $N \geq 1$ we have the bounds
\begin{equation}\label{upper bounds}
\left|m(P_D)-E_1(N;n,D) \right| \leq \frac{\Gamma(3/4)}{3}\frac{G(n,D)}{N^{3/4}} \quad and \quad
\left|m(P_D)-E_2(N;n,D) \right| \leq 2\sqrt[4]{2}\frac{G(n,D)}{\sqrt{N}},
\end{equation}
where
\begin{equation}\label{defn E_1}
E_1(N;n,D) = \log c(D) -\frac{1}{2}\sum_{j=1}^{N}\frac{1}{j}\sum_{k=0}^j\binom{j}{k} \frac{(-1)^k a(n,k,D)}{c(D)^{2k}}
\end{equation}
and
\begin{equation}\label{defn E_2}
E_2(N;n,D) = \log c(D) - \frac{1}{2} -\frac{1}{2}\sum_{j=1}^{N}\frac{2j+1}{j(j+1)}\sum_{k=0}^j\binom{j+k}{k}\binom{j}{k}
\frac{(-1)^k a(n,k,D)}{c(D)^{2k}}.
\end{equation}
\end{thm}

The bound we derive for $G(n,D)$ will be given in terms of the $J$-Bessel function, see formula \eqref{G(n,D) defin}.  However,
we note here that the bound for $G(n,D)$ is elementary.  In particular, Theorem \ref{thm:series}
leads to an explicit computational means by which one can estimate $m(P_{D})$ as accurately as one may wish.

We will derive explicit bounds for the tail of the series in \eqref{measure formula 2} for all $\ell \geq 2$ and
for all $D \neq r(1,1,\cdots, 1)$.  We refer the reader to Section 6 for the statements.  In the course of the proof
it becomes clear why this sole $D$ is singled out; it is the only instance where an application of the Cauchy-Schwarz
inequality is an equality rather than a strict inequality.

As is evident from equations \eqref{upper bounds} through \eqref{defn E_2} and the statements of our main theorems, the approximating
sum $E_1(N;n,D)$ is somewhat simpler and with a better rate of convergence to $m(P_D)$ than the sum $E_2(N;n,D)$. However, we find the estimate
to $m(P_D)$ by $E_2(N;n,D)$ to be theoretically interesting as well.  Indeed,  when combining the various expressions for
$m(P_D)$ derived above one has a potential source of combinatorial identities amongst weighted series of sums of binomial and multinomial coefficients.

Finally, we would like to emphasize that our main theorem is the first result of which we are aware which gives an explicit expression of the (logarithmic) Mahler measure of a
linear form in terms of an absolutely convergent series which involves elementary quantities, such as binomial and multinomial coefficients. The explicit and effective
upper bound for the approximation of $m(P_D)$ by a partial sum of this series provides a tool which can be used in estimating  the size of $m(P_D)$,
hence the canonical height of the divisor $\mathcal{D}$ of $P_D$; see \cite{Ma00} and the discussion in Section \ref{subsect: canonical heights}.

\subsection{Outline of the proofs}

In general terms, the analysis of the present paper involves a detailed investigation of the general
Kronecker-type limit formula which was proved in \cite{CJS20}.  In \cite{CJS20}, we considered a general
smooth K\"ahler manifold $X$ with divisor $\mathcal{D}$ which was assumed to be smooth up to codimension two.  For
this article, we take $X$ to be $n$-dimensional complex projective space $\CC\PP^{n}$ for $n\geq 2$ and $\mathcal{D}$
to be a hyperplane, meaning the divisor of a degree one polynomial $P_{D}(z)$.  We equip
$\CC\PP^{n}$ with its natural Fubini-Study metric.

With this setup, we employ a representation of the associated Green's function in terms of the heat kernel;
see \cite{HI02}.  The spectral expansion of the Green's function
also can be computed explicitly; see \cite{Lu98}.  In order to evaluate the Kronecker-type limit function
as in \cite{CJS20}, we need to integrate the Green's function on $\CC\PP^{n}$ along a hyperplane.  In
doing so, the evaluation of such integrals amounts to the Radon transform on projective space for which we
use results from \cite{Gr83}.  Ultimately, we are able to express the log-norm of the polynomial $P_{D}$ as an
absolutely convergent series of Jacobi polynomials.  The evaluation of the Mahler measure then reduces to the
problem of evaluating certain integrals involving Jacobi polynomials, which yields the results stated above.  The
different expressions for the Mahler measure $m(P_{D})$ amount to various identities involving Jacobi polynomials.

We wish to emphasize here that our main result holds for all linear polynomials provided $n \geq 3$; all
our results, except possibly \eqref{measure formula 2} with $\ell\geq 3$ and $D=r(1,1,\ldots,1)$ for some $r\neq 0$, hold when the the divisor of $P_{D}(z)$
intersects the domain of integration in \eqref{Mahler_definition}.
Previous authors such as \cite{Sm81} used techniques of complex analysis to obtain their results.  Their computations are important and
interesting, but are limited because of the logarithmic-type singularities which naturally occur.
From out point of view, such singularities are $L^{2}$, hence can be addressed when using real analytic methods.

\subsection{Organization of the paper}

In Section 2 we state additional notation and recall relevant results from the literature.  In Section 3
we will recall the general Kronecker-type limit formula from \cite{CJS20}, which holds for a reasonably general
K\"ahler manifold $X$, and make the result explicit in the case $X=\CC\PP^{n}$.  In Section 4
we study the results from Section 3 and obtain various expressions for $\log \Vert P_{D}\Vert_\mu$ where the norm is with respect to
the Fubini-Study metric.  In Section 5 we derive a change of variables formula which is an important ingredient in the
proof of our main result, carried out in Section 6. We conclude with Section 7 where we present several comments
regarding the analysis of this article.

\section{Preliminaries}

In this section we will introduce the necessary notation and prove some intermediate results related to the representation of the
resolvent kernel on $\CC\PP^{n}$ and its associated Kronecker limit formula as proved in \cite{CJS20}.

\subsection{Some special functions}

For any non-negative integers $\alpha$, $\beta$ and $j$, we let
$P_j^{(\alpha,\beta)}$ denote the Jacobi polynomial, which is defined for $x \in (-1,1)$ by
\begin{equation}\label{Jacobi_poly}
P_j^{(\alpha,\beta)}(x):=\frac{(-1)^j}{2^j j!}(1-x)^{-\alpha}(1+x)^{-\beta}\frac{d^j}{dx^j}\left[ (1-x)^{\alpha + j} (1+x)^{\beta + j}\right].
\end{equation}
If $\alpha=\beta=0$, then the Jacobi polynomials specialize to the Legrendre polynomials, which can
given by
$$
P_{j}(x) = \frac{1}{2^{j}j!}\frac{d^{j}}{dx^{j}}(x^{2}-1)^{j}.
$$
Many fascinating properties of Jacobi and Legrendre polynomials are developed in the classical text \cite{Sz74}.  In the
present article, we will use the following bound which we quote from \cite{HS14}.  For any $j \geq 1$ and for all $x\in[-1,1]$,
we have that
\begin{equation}\label{poly_bound1}
(1-x^2)^{\tfrac{1}{4}}|P_j(x)|\leq \sqrt{4/\pi}(2j+1)^{-\tfrac{1}{2}}.
\end{equation}
This bound is referred to as the the sharp form of the Bernstein's inequality for Legendre polynomials $P_j=P_j^{(0,0)}$; see Theorem 3.3 of \cite{Sz74}, \cite{Lo82/83}, or the discussion on page 228 of \cite{HS14}.
More generally, we will use the main theorem of \cite{HS14} which gives the uniform upper bound
\begin{equation}\label{poly_bound2}
(1-x^2)^{\tfrac{1}{4}}\left(\frac{1-x}{2}\right)^{(m-1)/2}|P^{(m-1,0)}_j(x)|\leq 12\cdot (2j+m)^{-\tfrac{1}{2}},
\end{equation}
which holds for all positive integers $j$ and $m$ and for $x\in [-1,1]$; see Theorem 1.1 and subsequent discussion on page 228 of \cite{HS14}.

For complex numbers $\nu$ and $z$ with $|\arg z|<\pi$, the Bessel function of the first kind $J_\nu(z)$  is defined by absolutely convergent power series
$$
J_\nu(z):=\frac{z^\nu}{2^\nu}\sum_{k=0}^{\infty}\frac{(-1)^k}{2^{2k}k!\Gamma(\nu+k+1)}z^{2k}.
$$
For non-integral complex $\nu$ and any complex $z$ with $|\arg z|<\pi$ one defines
the Bessel function of the second kind $Y_\nu(z)$ by $Y_\nu(z):=(\sin \pi \nu)^{-1}(\cos(\pi\nu)J_\nu(z)-J_{-\nu}(z))$; when $\nu=n$ is a non-negative integer,
then
\begin{align*}
\pi Y_n(z):&=2J_n(z)\log\left(\frac{z}{2}\right)-\sum_{k=0}^{n-1}\frac{(n-k-1)!}{k!}\left(\frac{z}{2}\right)^{2k-n}\\&-
\sum_{k=0}^{\infty}\frac{(-1)^k(z/2)^{2k+n}}{k!(k+n)!}\left(\frac{\Gamma'}{\Gamma}(k+1)+\frac{\Gamma'}{\Gamma}(k+n+1)\right),
\end{align*}
with the convention that the empty sum when $n=0$ is zero.
A thorough analysis of Bessel functions and functions associated with them can be found in the seminal book \cite{Wa66}.
The article \cite{Kr06} contains very explicit pointwise bounds for Bessel functions.
For our purposes, we will use the inequality from 7.31.2 \cite{Sz74} which states that
\begin{equation}\label{J0 basic ineq}
\vert J_0(2x)\vert \leq (\max \{1,\vert \pi x\vert \})^{-\tfrac{1}{2}}
\end{equation}

\subsection{Complex projective space}

Let $\CC\PP^n$ denote the $n$-dimensional complex projective space with the usual projective coordinates $(\mathcal{Z}_{0},\cdots, \mathcal{Z}_{n})$.  If $U$ is any open set in $\CC\PP^n$ and $z:U\rightarrow \CC^{n+1}\setminus\{0\}$ a holomorphic lifting of $U$, so a holomorphically varying choice of homogeneous coordinates for $z$, then the local K\"ahler potential is given by $\rho(z)=\log\Vert z\Vert^2=\log(|\mathcal{Z}_0|^2 + \cdots + |\mathcal{Z}_n|^2)$.  The K\"ahler $(1,1)$ form $\frac{i}{2}\partial_z\partial_{\bar{z}}\rho$ will be denoted by $\omega$ and we equip $\CC\PP^n$ with the Fubini-Study metric $\mu=\mu_{FS}$ associated to $\omega$.
The Fubini-Study distance between two points $z,w\in\CC\PP^{n}$ will be denoted by $d_{\mathrm{FS}}(z,w)$. It is given by the formula
\[
\cos(d_{\mathrm{FS}}(z,w))=\frac{|\langle z,w\rangle|}{\sqrt{\langle z,z\rangle}\sqrt{\langle w,w\rangle}}
\]
where, if $z=(\mathcal{Z}_0, \dots,\mathcal{Z}_n)$ and $w=(\mathcal{W}_0,\dots , \mathcal{W}_n)$, then $\langle z,w\rangle= z\cdot{^t\overline{w}}=\mathcal{Z}_0\overline{\mathcal{W}_0}+\cdots \mathcal Z_n\overline{\mathcal{W}_n}$.

Occasionally, our computations
will be on the affine chart where $\mathcal{Z}_{0} \neq 0$, so then we will consider the affine coordinates $(z_1,\ldots,z_n)$.
Then the local K\"ahler potential takes the form
$\rho_0(z)=\log(1+|z_1|^2+\ldots+|z_n|^2)$.

Let $P_{D}(z)$ denote any homogenous polynomial with divisor $\mathcal D$.  We denote by $\Vert P_{D}(z)\Vert^{2}_{\mu}$
the $\log$-norm of the polynomial $P_D$ with respect to $\mu$.  The formula for $\Vert P_{D}(z)\Vert^{2}_{\mu}$
is
\begin{equation}\label{log norm of P}
\log \Vert P_{D}(z)\Vert^{2}_{\mu}=\log |P_D(z)|^{2}-\textrm{\rm deg}(P_{D}) \rho(z)
\end{equation}
for $z\in\CC\PP^n \setminus \mathcal{D}$.  If $z$ approaches $\mathcal D$ transversally, then $\log \Vert P_{D}(z)\Vert_{\mu}^{2}$ has a logarithmic
singularity which is $L^{1}$ integrable. For the sake of brevity, we may omit the subscript $\mu$.

Let $\Delta_{\CC\PP^n}$ signify the corresponding Laplacian $\Delta_{\CC\PP^n}$ which acts on smooth functions on $\CC\PP^n$.
An eigenfunction of the Laplacian $\Delta_{\CC\PP^n}$  is an \it a priori \rm $C^{2}$ function $\psi_{j}$
which satisfies the equation
$$
\Delta_{\CC\PP^n} \psi_{j} + \lambda_{j} \psi_{j}= 0
$$
for some constant $\lambda_{j}$, which is the eigenvalue associated to $\psi_{j}$. The spectrum $\{\lambda_j\}_{j\geq0}$ of $\Delta_{\CC\PP^n}$ is well known;
see, for example, \cite{BGM71} or \cite{Lu98}.  Classically, $\lambda_0=0$ where the eigenfunction is the appropriately normalized positive constant function.

Let $\mathrm{vol}_{\mu}(\CC\PP^n)$ denote the volume of $\CC\PP^n$, meaning the integral over $\CC\PP^n$ of the volume form $\mu^n$.
In our normalizations, we have that
$$
\mathrm{vol}_{\mu}(\CC\PP^n) = \frac{\pi^n}{n!}.
$$
Additionally, we have that $\lambda_j=4j(j+n)$ for all $j\geq 1$.  Let $H_{j,j}(n+1)$ be the vector space of eigenfunctions with eigenvalue $\lambda_j$,
Then the dimension $N_{j}$ of $H_{j,j}$ is
$$
N_j=\binom{n+j}{j}^2-\binom{n+j-1}{j-1}^2 = \frac{(n+2j)((n+j-1)!)^2}{n!(n-1)!(j!)^2}.
$$
Moreover, as discussed in section 1 of \cite{Gr83}, the Hilbert space $L^2(\CC\PP^n)$ of all square integrable functions
on $\CC\PP^n$, with respect to the volume form $\mu^n$, has the orthogonal decomposition
$$
L^2(\CC\PP^n)= \bigoplus_{j=0}^{\infty}H_{j,j}(n+1)
$$
into finite dimensional  subspaces $H_{j,j}(n+1)$ consisting of eigenfunctions of the Laplacian with the corresponding eigenvalue $4j(j+n)$. Each subspace $H_{j,j}(n+1)$ is an irreducible representation of the unitary group $\mathbf{U}(n+1)$ and  they are distinct.
More precisely, elements of $H_{j,j}(n+1)$ are homogeneous harmonic polynomials of degree $j$ in the variables $\mathcal{Z}_0,...,\mathcal{Z}_n$
and $\overline{\mathcal{Z}}_0,..., \overline{\mathcal{Z}}_n$.  As is standard, we may assume that coefficients of those harmonic polynomials are real
so then any eigenfunction evaluated at real values of its variables is itself real-valued.

\subsection{The Radon transform}

Let $f$ be a continuous function on $\CC\PP^{n}$, and let $H$ be any hyperplane in $\CC\PP^n$.  The Radon transform
of $f$ evaluated at $H$, which we denote by $Rf(H)$, is defined by
$$
Rf(H)=\int\limits_H f(w)\mu_{H}(w)
$$
where $\mu_{H}(w)$ is the Fubini-Study volume element induced on $H$ from the Fubini-Study metric on $\CC\PP^{n}$.  Denote the Grassmannian of
hyperplanes in $\CC\PP^n$ by $(\CC\PP^{n})^{\ast}$.  Recall that $(\CC\PP^{n})^{\ast}$ is non-canonically isomorphic to $\CC\PP^{n}$.  Let us
make a choice regarding this isomorphism.  Quite simply, the point $(\mathcal{W}_0,\mathcal{W}_1,\ldots,\mathcal{W}_n)\in \CC\PP^n$
is identified with the hyperplane
$$
\{(\mathcal{Z}_0,\mathcal{Z}_1,\ldots,\mathcal{Z}_n)\in\CC\PP^n:
\mathcal{Z}_0\mathcal{W}_0+ \mathcal{Z}_1\mathcal{W}_1+\ldots+\mathcal{Z}_n\mathcal{W}_n=0\}.
$$
As such, we can view the Radon transform $Rf(H)$ of $f$ as a function on $\CC\PP^n$.

As proved in \cite{Gr83}, by Schur's Lemma the Radon transform $R$ acts on $H_{j,j}(n+1)$ by scalar multiplication by $c(j,n)=c_n\cdot \frac{(-1)^j j!}{(j+n-1)!}$
where $c_n$ is certain normalizing factor depending solely on the dimension $n$.  In our setting, the normalizing factor
can be easily computed by evaluating the Radon transform of the $L^{2}$-normalized constant eigenfunction
$\psi_0(w)=\frac{1}{\sqrt{\mathrm{vol}_{\mu}(\CC\PP^n)}}$ and taking $H$ to be the (affine) hyperplane $z_1=0$.
In this case $c(0,n)=c_n\cdot\frac{1}{(n-1)!}$, so then
$$
\int\limits_{H}\frac{1}{\sqrt{\mathrm{vol}_{\mu}(\CC\PP^n)}}\mu_{H}(w)=c_n\cdot\frac{1}{(n-1)!}
\cdot\frac{1}{\sqrt{\mathrm{vol}_{\mu}(\CC\PP^n)}}.
$$
Therefore, $c_n=(n-1)!\cdot \mathrm{vol}_{\mu}(\CC\PP^{n-1})=\pi^{n-1}$. If $\psi_j\in H_{j,j}(n+1)$, where $R$ acts by the scalar $c(j,n)$,  we will simply have $R\psi_j(H)=c(j,n)\psi_j(H)$ where we identify $(\CC\PP^n)^*$ with $\CC\PP^n$ as above.

In summary, we have the following formula.  With the notation and normalizations set above,
for any hyperplane $H$ in $\CC\PP^n$ and any eigenfunction $\psi_j\in H_{j,j}(n+1)$, one has that
\begin{equation}\label{Radon action}
\int\limits_{H}\psi_j(w)\mu_{H}(w)= \pi^{n-1}\cdot \frac{(-1)^j j!}{(j+n-1)!}\psi_j(H).
\end{equation}
Again, it is necessary to note that we have chosen an identification between
the hyperplane $H$, which is a point in $(\CC\PP^{n})^{\ast}$, and a point in $\CC\PP^{n}$, which
through a slight abuse of notation we also write as $H$.

\subsection{Heat kernel and Green's function}

The heat kernel $K_{\CC\PP^n}(z,w;t)$ associated to the Laplacian $\Delta_{\CC\PP^n}$ on $\CC\PP^n$ is the unique solution to the heat-equation
$$
\frac{\partial}{\partial t} K_{\CC\PP^n}(z,w;t) = \Delta_{\CC\PP^n}K_{\CC\PP^n}(z,w;t)
\,\,\,\,\,
\textrm{\rm \,\,\, for \,\,\, $t>0$ \,\,\, and \,\,\,$z,w \in\CC\PP^n$}
$$
such that for any continuous function $f$ on $\CC\PP^{n}$, one has
$$
\lim\limits_{t \rightarrow 0}\int\limits_{\CC\PP^{n}}K_{\CC\PP^n}(z,w;t)f(w) \mu_{\CC\PP^{n}}(w) =  f(z).
$$
As can be shown, $K_{\CC\PP^n}(z,w;t)$ depends only on $t>0$ and the Fubini-Study distance between $z$ and $w$;
see, for example, \cite{Lu98}.  The heat kernel admits a spectral expansion in terms of the eigenfunctions
$\psi_j\in H_{j,j}(n+1)$.  Namely, one has the formula that
\begin{equation} \label{heat kernel spectral exp}
K_{\CC\PP^n}(z,w;t)=\sum_{\lambda_j\geq0}\psi_j(z)\overline{\psi_j(w)}e^{-\lambda_j t}= \frac{1}{\mathrm{vol}_{\mu}(\CC\PP^n)} + \sum_{\lambda_j>0}\psi_j(z)\overline{\psi_j(w)}e^{-\lambda_j t}.
\end{equation}
In \cite{HI02} it is proved that
\begin{equation} \label{heat kernel jacobi pol exp}
K_{\CC\PP^n}(z,w;t)= \frac{1}{\mathrm{vol}_{\mu}(\CC\PP^n)} +\frac{1}{\pi^n}\sum_{j=1}^{\infty}(2j+n)\frac{(j+n-1)!}{j!}P_j^{(n-1,0)}(\cos(2r))e^{-4j(j+n) t},
\end{equation}
where $r=d_{\textrm{FS}}(z, w)$ is the Fubini-Study distance and $P_j^{(\alpha,\beta)}$ is the Jacobi polynomial defined in \eqref{Jacobi_poly}.
Actually, the transition from \eqref{heat kernel spectral exp} to \eqref{heat kernel jacobi pol exp} is based on stronger
results.  Indeed, it is proved that
\begin{equation}\label{eigenfunctions identity}
\sum_{\lambda_j = 4j(n+j)}\psi_j(z)\overline{\psi_j(w)} =
\frac{(2j+n)}{\pi^n}\frac{(j+n-1)!}{j!}P_j^{(n-1,0)}(\cos(2r));
\end{equation}
see Theorem 1 of \cite{Lu98} as well as Theorem 1 and the preceding discussion in \cite{HI02}, keeping in mind
the notation conventions employed in the present article.

The Green's function $G_{\CC\PP^n}(z,w;s)$ on $\CC\PP^n$ is the integeral kernel of the right inverse to the operator $\Delta_{\CC\PP^n}+s(1-s)$ on
$L^{2}(\CC\PP^n)$ for $s \in \CC$.  In order for such an inverse to exist, it is necessary to assume that
$s(1-s)$ is not equal to an eigenvalue of $\Delta_{\CC\PP^n}$.  As discussed in Section 5 of \cite{CJS20} and references therein,
the Green's function $G_{\CC\PP^n}(z,w;s)$ and  the heat kernel $K_{\CC\PP^n}(z,w;t)$ are related by the
formula
\begin{equation} \label{green in terms of heat}
G_{\CC\PP^n}(z,w;s)-\frac{1}{\textrm{\rm vol}_\mu (\CC\PP^n)}\frac{1}{s^2}= \int\limits_0^{\infty}\left( K_{\CC\PP^n}(z,w;t) -
\frac{1}{\textrm{\rm vol}_\mu (\CC\PP^n)}\right)e^{-s^2t} dt.
\end{equation}
The identity (\ref{green in terms of heat}) holds for $s \in \CC$ with
$\Re(s^2)>0$ and for all distinct points $z,w\in\CC\PP^n$.  However, one can use the spectral
expansion of the heat kernel in order to obtain a meromorphic continuation of (\ref{green in terms of heat})
to all $s \in \CC$.

\section{A Kronecker limit formula}

The following result from \cite{CJS20} is an analogue of the classical Kronecker's limit formula, which
is stated here in the setting of projective space.

\begin{thm}\label{thm4 from cjs} Let $\mathcal D$ be the divisor of a polynomial $P_D$ on $\CC\PP^n$, and assume that
$\mathcal D$ is smooth up to codimension two in $\CC\PP^n$. Then there exist constants $c_0$ and $c_1$
such that for $z\notin \mathcal D$ we have that
\begin{equation}\label{e.Gseries_KLF}
\int\limits_{\mathcal D}G_{\CC\PP^n}(z,w;s)\mu_{\mathcal D}(w) = \frac{\textrm{\rm vol}_\mu (\mathcal{D})}{\textrm{\rm vol}_\mu (\CC\PP^n)}\frac{1}{s^2} +
c_{0}\log \Vert P_{D}(z)\Vert^{2}_{\mu} +c_{1}+ O(s)
\,\,\,\,\,
\textrm{as $s \rightarrow 0$.}
\end{equation}
\end{thm}

Let us now consider the case when $P_D$ is a linear polynomial in $n+1$ projective coordinates $(\mathcal{Z}_{0},\ldots,\mathcal{Z}_n)$ of $\CC\PP^n$.
 With this, Theorem \ref{thm4 from cjs} becomes the following result.

\begin{prop}\label{prop:Thm4 improved}
Let $P_D(z)=P_D(\mathcal{Z}_{0},\ldots,\mathcal{Z}_n)=\mathcal{W}_0\mathcal{Z}_{0}+\mathcal{W}_1\mathcal{Z}_1+...+\mathcal{W}_n\mathcal{Z}_n$ be the linear polynomial in $n+1$
complex projective coordinates variables with the divisor $\mathcal D$.  Let $H_{n}$ denote the $n$-th harmonic number.  Then, for $z\notin \mathcal D$ we have that
\begin{align}\nonumber
\int\limits_{\mathcal D}G_{\CC\PP^n}(z,w;s)\mu_{\mathcal D}(w) &= \frac{\textrm{\rm vol}_\mu (\mathcal D)}{\textrm{\rm vol}_\mu (\CC\PP^n)}\frac{1}{s^2}- \frac{1}{4\pi}\log| P_{D}(z)|^{2}
\\& \label{e.Gseries_KLF new} + \frac{1}{4\pi}\left(\log||D||^2+\rho(z) -H_n\right) + O(s)
\,\,\,\,\,
\textrm{as $s \rightarrow 0$.}
\end{align}
\end{prop}

\begin{proof}
Set
$$
G_{\CC\PP^n}(z,w)=\lim_{s\to 0}\left(G_{\CC\PP^n}(z,w;s)\mu_{\mathcal D}(w) - \frac{1}{\textrm{\rm vol}_\mu (\CC\PP^n)}\frac{1}{s^2} \right).
$$
Then
$G_{\CC\PP^n}(z,w)$ is the integral kernel which inverts the action of the Laplacian when restricted to the subspace
of $L^{2}(\CC\PP^{n})$ that is orthogonal to the constant functions.
Formula \eqref{e.Gseries_KLF} in Theorem \ref{thm4 from cjs} can be written as
$$
\int\limits_{\mathcal D}G_{\CC\PP^n}(z,w)\mu_{\mathcal D}(w) = c_{0}\log \Vert P_{D}(z)\Vert^{2}_{\mu} +c_{1}.
$$
If we now integrate over $\CC\PP^n$ with respect to $\mu_{\CC\PP^{n}}(z)$, we get that
$$
0=c_0\int\limits_{\CC\PP^n}\log \Vert P_{D}(z)\Vert^{2}_{\mu}\mu_{\CC\PP^{n}}(z) + c_1\textrm{\rm vol}_\mu (\CC\PP^n).
$$
In our normalizations, we have that $c_0=-1/4\pi$; see, for example, page 94 of \cite{Fo76} and page
10 of \cite{La88} as well as page 338 of \cite{JK98}.   As a result, we have
\begin{equation}\label{c1 staring formula}
c_1=\frac{1}{4\pi}\cdot\frac{1}{\textrm{\rm vol}_\mu (\CC\PP^n)}\int\limits_{\CC\PP^n}\log \Vert P_{D}(z)\Vert^{2}_{\mu}\mu_{\CC\PP^{n}}(z).
\end{equation}
It is left to evaluate the integral on the right-hand side of \eqref{c1 staring formula}.

Set $D = (\mathcal{W}_0, \mathcal{W}_1,\ldots,\mathcal{W}_n)$, so then $P_D(z) = D\cdot {^tz}$.  One can choose an
element $\gamma \in \mathbf{U}(n+1)$ so that $ D \gamma= (c,0,\ldots, 0)$.  Through a rescaling of $w$ by a complex
number of absolute value one, we may assume that
$$
c = \Vert D \Vert.
$$
Let $z=\tilde{z} {^t{\gamma}}$, where $\tilde{z}=(\tilde{\mathcal{Z}}_{0}, \tilde{\mathcal{Z}}_{1},\ldots, \tilde{\mathcal{Z}}_{n})$.
On the affine chart $\tilde{\mathcal{Z}}_{0} \neq 0$, the
polynomial $P_{D}(z)$ is simply $P_D(z)=P_D(\tilde{z}{^t\gamma}) = D\gamma {^t\tilde{z}}=c\tilde{\mathcal{Z}}_{0}$.

By the $\mathbf{U}(n+1)$ invariance of the Fubini-Study volume form
\[
\int\limits_{\CC\PP^n}\log \Vert P_D(z)\Vert^{2}_{\mu}\mu_{\CC\PP^{n}}(z)
= \int\limits_{\CC\PP^n}\log \Vert P_D(\tilde{z}{^t{\gamma}})\Vert^{2}_{\mu}\mu_{\CC\PP^{n}}(\tilde{z}).
\]
By  \eqref{log norm of P}  we have
\[
\log \Vert P_{D}(\tilde{z}{^t\gamma})\Vert^{2}_{\mu}=\log |P_D(\tilde{z}{^t\gamma})|^{2}-\textrm{\rm deg}(P_{D}) \rho(\tilde{z}{^t\gamma}).
\]
Here  we have
$\log|P_D(\tilde{z}{^t\gamma})|^2=\log|c\tilde{\mathcal{Z}}_0|^2=\log(c^2)+\log|\tilde{\mathcal{Z}}_0|^2$, $\deg(P_D)=1$,
and $\rho(\tilde{z}{^t\gamma})=\log\Vert\tilde{z}{^t\gamma}\Vert^2=\log\Vert \tilde{z}\Vert^2=\log(|\tilde{\mathcal{Z}}_0|^2+\cdots+|\tilde{\mathcal{Z}}_n|^2)$.
Therefore
\[
\log \Vert P_{D}(\tilde{z}{^t\gamma})\Vert^{2}_{\mu}=\log(c^2)+\log|\tilde{\mathcal{Z}}_0|^2-\log(|\tilde{\mathcal{Z}}_0|^2+\cdots+|\tilde{\mathcal{Z}}_n|^2)=\log\Vert D\Vert^2-\rho_0(\tilde{z})
\]
and
\begin{equation}
\int\limits_{\CC\PP^n}\log \Vert P_D(z)\Vert^{2}_{\mu}\mu_{\CC\PP^{n}}(z)=\log\Vert D \Vert^2\textrm{\rm vol}_\mu (\CC\PP^n) - \int\limits_{\CC\PP^n}\rho_0(\tilde{z})\mu_{\CC\PP^{n}}(\tilde{z}). \label{integral of polynom}
\end{equation}

The last integral can be evaluated using polar coordinates on the affine chart $\tilde{\mathcal{Z}}_{0}\neq 0$, which
can be viewed as $\CC^{n}$.  Indeed, if we let $S^{2n-1}$ denote the unit sphere in $\CC^{n}$, then
\begin{align*}
\int\limits_{\CC\PP^n}\rho_0(\tilde{z})\mu_{\CC\PP^{n}}(\tilde{z})&=
\text{\rm vol}(S^{2n-1})
\int\limits_{0}^{\infty}\log (1+\rho^{2})\frac{\rho^{2n-1}}{(\rho^{2}+1)^{n+1}}d\rho
\\&=
n \textrm{\rm vol}_\mu (\CC\PP^n)
\int\limits_{1}^{\infty}\log (t)\frac{(1-t)^{n-1}}{t^{n+1}}dt.
\end{align*}
The last integral can be evaluated by using formula 4.253.3 from \cite{GR07} with $u=1$, $\mu=n$ and $\lambda=n+1$.
With this, we obtain that
$$
\int\limits_{\CC\PP^n} \rho_0(\tilde{z})\mu_{\CC\PP^{n}}(\tilde{z})=  \textrm{\rm vol}_\mu (\CC\PP^n)\left(\frac{\Gamma'}{\Gamma}(n+1)- \frac{\Gamma'}{\Gamma}(1)\right) = \textrm{\rm vol}_\mu (\CC\PP^n)H_n.
$$
Combining this evaluation with \eqref{integral of polynom} and \eqref{c1 staring formula} we arrive at the formula that
$$
c_1=\frac{1}{4\pi}\left(\log||D||^2 -H_n\right).
$$
By substituting this expression together with $c_0=-1/(4\pi)$ into \eqref{e.Gseries_KLF}, and employing \eqref{log norm of P} we have completed our proof of \eqref{e.Gseries_KLF new}.
\end{proof}

\begin{rem}\rm
An anonymous referee informed us that a different approach may be undertaken to derive the calculations in this section by relating $m(P_{D})$ to a computation of integrals of
Green's functions for the divisor $\mathcal{D}$; see Section \ref{subsect: canonical heights} for further discussion.
With this approach, the computation of constant $c_1$ in the above proposition is related to the proof of Proposition 5.3 from \cite{GS90}.
\end{rem}

\section{The $\log$-norm of a linear polynomial}\label{sec:log norm}

In this section we will refine further the results of Theorem \ref{thm4 from cjs} and Proposition \ref{prop:Thm4 improved} by expressing
the integral of the Green's function in terms of a certain series involving Jacobi polynomials.  The convergence issues will
be addressed by appealing to the bounds \eqref{poly_bound1} and \eqref{poly_bound2}.

\begin{prop}\label{prop:greens function evaluation}
 With the above notation, assume that $\mathrm{Re}(s^2)>0$,  $z\notin \mathcal D \cup\{D\}$. Then, we have the relation
\begin{align}\nonumber \label{green in terms of sum}
\int\limits_{\mathcal D}\left(G_{\CC\PP^n}(z,w;s)-\frac{1}{\textrm{\rm vol}_\mu (\CC\PP^n)}\frac{1}{s^2}\right)&\mu_{\mathcal D}(w)
\\= \frac{1}{\pi} \sum_{j=1}^{\infty} \frac{(2j+n)(-1)^j}{s^2+4j(j+n)} &P_j^{(n-1,0)}(\cos(2d_{\textrm{FS}}(z, {D}))).
\end{align}
The series in \eqref{green in terms of sum} converges uniformly and absolutely when $z$
and $s$ lie in compact subsets of the above specified regions.
\end{prop}

\begin{proof}
We start with the expression \eqref{green in terms of heat} for the Green's function in terms of the heat kernel
$K_{\CC\PP^n}(z,w;t)$ on $\CC\PP^n$ for $\mathrm{Re}(s^2)>0$. Using the equations \eqref{heat kernel spectral exp} and \eqref{heat kernel jacobi pol exp} the function under the integral sign in \eqref{green in terms of heat} can be written as
\begin{align*}
K_{\CC\PP^n}(z,w;t) - \frac{1}{\textrm{\rm vol}_\mu (\CC\PP^n)}&=\sum_{\lambda_j>0}\psi_j(z)\overline{\psi_j(w)}e^{-\lambda_j t} \\&= \frac{1}{\pi^n}\sum_{j=1}^{\infty}(2j+n)\frac{(j+n-1)!}{j!}P_j^{(n-1,0)}(\cos(2r))e^{-4j(j+n) t},
\end{align*}
where we have set $r=d_{\textrm{FS}}(z, w)$.  The above series is
uniformly bounded provided the distance between $z$ and $w$ is bounded away from zero, which is
satisfied for $z \notin \mathcal D$ since we will consider $w \in \mathcal D$.



 Therefore,
we can apply the Fubini-Tonelli theorem to get that
$$
\int\limits_{\mathcal D}\left(G_{\CC\PP^n}(z,w;s)-\frac{1}{\textrm{\rm vol}_\mu (\CC\PP^n)}\frac{1}{s^2}\right)\mu_{\mathcal D}(w)= \int\limits_0^{\infty}\int\limits_{\mathcal D}\left( \sum_{\lambda_j>0}\psi_j(z)\overline{\psi_j(w)}e^{-\lambda_j t}\right)\mu_{\mathcal D}(w) e^{-s^2t} dt.
$$
As stated above, we may assume that the eigenfunctions $\psi_j$ are homogeneous polynomials with real coefficients, so then
$$
\int\limits_{\mathcal D}\overline{\psi_j(w)}\mu_{\mathcal D}(w)=\int\limits_{\mathcal D}\psi_j(\overline{w})\mu_{\mathcal D}(w)= \int\limits_{\overline{\mathcal D}}\psi_j(w)\mu_{\overline{\mathcal D}}(w).
$$
In other words, integrating $\overline{\psi_j(w)}$ over $\mathcal D$ amounts to taking the Radon transform of the function $\psi_j$
 which belongs to the subspace $H_{j,j}(n+1)$, so the integral equals the integral of $\psi_{j}(w)$ over the
 conjugate hyperplane $\overline{\mathcal D}$. On the other hand, under our isomorphism $(\CC\PP^n)^*\simeq \CC\PP^n$ we have that $\overline{\mathcal{D}}$ corresponds to $\overline{D}$.
Let us make this precise. $\overline{\mathcal D}$ is the divisor $\overline{\mathcal{D}}=\{\overline{v}=(\overline{\mathcal{V}}_0, \dots,\overline{\mathcal{V}}_n) \mid P_D(\overline{v})=0\}$. Now $P_D(\overline{v})=D\cdot {^t\overline{v}}$ so $P_D(\overline{v})=0$ is equivalent to $D\cdot{^t\overline{v}}=0=\overline{D}\cdot{^tv}$. So under the isomorphism $(\CC\PP^n)^*\simeq \CC\PP^n$ we have $\overline{\mathcal{D}}$ corresponds to $\overline{D}$.


For fixed and positive $t$, we can use equation \eqref{Radon action} as applied to the hyperplane $\overline{\mathcal D}$ combined with the equation \eqref{eigenfunctions identity} to get the formula
\begin{align} \label{int over D}\nonumber
\int\limits_{\mathcal D} \sum_{\lambda_j>0}\psi_j(z)\overline{\psi_j(w)}&e^{-\lambda_j t}\mu_{\mathcal D}(w)= \sum_{j=1}^{\infty} \frac{\pi^{n-1}(-1)^j j!}{(j+n-1)!} \sum_{\lambda_j=4j(n+j)}\psi_j(z)\psi_j(\overline{\mathcal D})e^{-\lambda_j t} \nonumber \\
&=\sum_{j=1}^{\infty} \frac{\pi^{n-1}(-1)^j j!}{(j+n-1)!} \sum_{\lambda_j=4j(n+j)}\psi_j(z)\psi_j(\overline{ D})e^{-\lambda_j t} \nonumber\\
&=\sum_{j=1}^{\infty} \frac{\pi^{n-1}(-1)^j j!}{(j+n-1)!} \sum_{\lambda_j=4j(n+j)}\psi_j(z)\overline{\psi_j({ D})}e^{-\lambda_j t} \nonumber\\
&=\sum_{j=1}^{\infty} \frac{\pi^{n-1}(-1)^j j!}{(j+n-1)!} \frac{(2j+n)}{\pi^n}\frac{(j+n-1)!}{j!}P_j^{(n-1,0)}(\cos(2d_{\textrm{FS}}(z,D)))e^{-\lambda_j t} \nonumber\\
&=\frac{1}{\pi} \sum_{j=1}^{\infty} (2j+n)(-1)^j P_j^{(n-1,0)}(\cos(2d_{\textrm{FS}}(z,{D})))e^{-4j(j+n)t}.
\end{align}
As described above, the eigenfunctions are appropriately scaled so that
$\psi_j(\overline{D})=\overline{\psi_j( D)}$.
Now, multiply \eqref{int over D} by $e^{-s^{2}t}$ and then integrate with respect to $t$ for $t>0$.  After interchanging the sum and the integral over $t$,
we deduce that
\begin{align*}
\int\limits_0^{\infty}\int\limits_{\mathcal D}&\left( \sum_{\lambda_j>0}\psi_j(z)\overline{\psi_j(w)}e^{-\lambda_j t}\right)\mu_{\mathcal D}(w) e^{-s^2t} dt=\frac{1}{\pi} \sum_{j=1}^{\infty} \frac{(2j+n)(-1)^j}{s^2+4j(j+n)} P_j^{(n-1,0)}(\cos(2 d_{\textrm{FS}}(z, {D}))).
\end{align*}
The absolute convergence of each series in the above discussion is confirmed using the bound \eqref{poly_bound2}.
With all this, the proof of \eqref{green in terms of sum} is complete.
\end{proof}

By combining Proposition \ref{prop:Thm4 improved} and Proposition \ref{prop:greens function evaluation} we arrive at the identity
\begin{align*}
\frac{1}{\pi} \sum_{j=1}^{\infty} \frac{(2j+n)(-1)^j}{s^2+4j(j+n)} &P_j^{(n-1,0)}(\cos(2d_{\textrm{FS}}(z, {D})))=- \frac{1}{4\pi}\log \vert P_{D}(z)\vert^{2} \\&+ \frac{1}{4\pi}\left(\log||D||^2+\rho(z) -H_n\right) + O(s)
\,\,\,\,\,\textrm{\rm as $s \rightarrow 0$ with $\Re(s^2)>0$.}
\end{align*}
The series of the left-hand side of the above equation is a holomorphic function of $s$ in the region $\Re(s^2)>0$ which can be analytically continued to $s=0$, since the integral on the left-hand side of the equation \eqref{green in terms of sum} is analytic at $s=0$. The uniqueness of Taylor series representation of analytic function  implies the identity
\begin{equation} \label{starting identity}
\log \vert P_{D}(z)\vert^{2}=\log||D||^2+\rho(z) -H_n - \sum_{j=1}^{\infty} \frac{(2j+n)(-1)^j}{j(j+n)} P_j^{(n-1,0)}(\cos(2d_{\textrm{FS}}(z, {D}))).
\end{equation}
The bound \eqref{poly_bound2} immediately implies that
the series in \eqref{starting identity} uniformly and absolutely for $w\in \mathcal D$ and $z$ in compact subsets of $\CC\PP^{n}\setminus (\mathcal D\cup \{D\})$.

The next proposition considers the sum of Jacobi polynomials.

\begin{prop}\label{prop_Legendra_formulas} With the notation as above, for $r=d_{\textrm{FS}}(z,{D})\neq0$ and any positive integer $\ell$, one
has that
\begin{align*}
\sum_{j=1}^{\infty} \frac{(2j+\ell)(-1)^j}{j(j+\ell)} P_j^{(\ell-1,0)}(\cos(2r))+H_\ell&=-\frac{d}{d\nu}P_{\nu}(\cos(2r))\Big|_{\nu=0}\\&=-\frac{d}{d\nu}F(-\nu,\nu+1;1;1-\cos^2 r)\Big|_{\nu=0}\\
&=-\int\limits_0^\infty \left(\pi Y_1(x)+ \frac{2}{x}J_0(x)\right)J_0(x\sin r)dx,
\end{align*}
where $P_{\nu}(x)$ is the Legendre function, $F(-\nu,\nu+1;1;1-\cos^2 r)$ is the classical hypergeometric function and $Y_1$ and $J_0$ are the Bessel functions.
\end{prop}

\begin{proof}
By starting with the formula 8.961.8. from \cite{GR07} we see that for any $\ell \geq 2$ and $x\in [-1,1)$
\begin{align*}
\sum_{j=1}^{\infty}(-1)^j\frac{2j+\ell}{j(j+\ell)}P_j^{(\ell-1,0)}(x)&=\sum_{j=1}^{\infty}\frac{(-1)^j}{j}P_j^{(\ell,0)}(x)- \sum_{j=1}^{\infty}\frac{(-1)^j}{(j+\ell)}P_{j-1}^{(\ell,0)}(x)\\&=\sum_{j=1}^{\infty}(-1)^j\frac{2j+\ell+1}{j(j+\ell+1)}P_j^{(\ell,0)}(x) + \frac{1}{\ell+1}.
\end{align*}
By replacing $\ell$ with $\ell-1$, we get that
\begin{equation}\label{change_of_ell}
\sum_{j=1}^{\infty}(-1)^j\frac{2j+\ell}{j(j+\ell)}P_j^{(\ell-1,0)}(x)=\sum_{j=1}^{\infty}(-1)^j\frac{2j+\ell-1}{j(j+\ell-1)}P_j^{(\ell-2,0)}(x)-\frac{1}{\ell}.
\end{equation}
Let us now repeat this process $\ell-1$ times, after which we get for any positive integer $\ell$ and $r\neq 0$ the formula
\begin{equation} \label{first sum}
\sum_{j=1}^{\infty}(-1)^j\frac{2j+\ell}{j(j+\ell)}P_j^{(\ell-1,0)}(\cos(2r))+H_\ell=\sum_{j=1}^{\infty}(-1)^j\frac{2j+1}{j(j+1)}P_j(\cos(2r))+1,
\end{equation}
where $P_n$ denotes the Legendre polynomial.

From formula 8.793 of \cite{GR07} we have, for any $\nu \notin \mathbb{Z}$ and $x \in (-1,1]$, the expression
\begin{equation}\label{Pseries}
\sum\limits_{k=0}^{\infty}(-1)^{k}\left(\frac{1}{\nu-k}-\frac{1}{k+\nu+1}\right)P_{k}(x) = \frac{\pi P_{\nu}(x)}{\sin(\nu x)}.
\end{equation}
From (\ref{Pseries}), we can write
$$
\sum\limits_{k=1}^{\infty}(-1)^{k}\left(\frac{1}{\nu-k}-\frac{1}{k+\nu+1}\right)P_{k}(x) = \frac{\pi P_{\nu}(x)}{\sin(\nu x)} - \left(\frac{1}{\nu} - \frac{1}{\nu+1}\right).
$$
By letting $\nu$ approach zero, one obtains
$$
\sum\limits_{k=1}^{\infty}(-1)^{k}\left(\frac{1}{-k}-\frac{1}{k+1}\right)P_{k}(x) = \frac{d}{d\nu}P_{\nu}(x)\Big|_{\nu=0}+1,
$$
or
$$
\sum\limits_{k=1}^{\infty}(-1)^{k}\left(\frac{1}{k}+\frac{1}{k+1}\right)P_{k}(x) = -1-\frac{d}{d\nu}P_{\nu}(x)\Big|_{\nu=0}.
$$
By letting $x= \cos(2r)$, where $r=d_{\textrm{FS}}(z,{D})$, and combining the last equation with \eqref{first sum} we arrive at the identity
$$
\sum_{j=1}^{\infty}(-1)^j\frac{2j+\ell}{j(j+\ell)}P_j^{(\ell-1,0)}(\cos(2r))+H_\ell= -\frac{d}{d\nu}P_{\nu}(\cos(2r))\Big|_{\nu=0},
$$
which holds for any $\ell\geq 1$ and $r\neq 0$.  With this, we have proved the first equality which was claimed in the statement of the Proposition.

For convenience, let us recall that $r = \textrm{d}_{\textrm{FS}}(z,{D})$ and $\cos^{2}(r) = \frac{1}{2}(\cos(2r)+1).$ By
applying formula 8.751.1 with $m=0$ from \cite{GR07} we obtain that
$$
P_{\nu}(\cos(2r))= P_{\nu}(2\cos^2 r -1)=F(-\nu,\nu+1;1;1-\cos^2 r).
$$
Further, if we employ formula 6.512.1 from \cite{GR07} (where, in their notation, we take $\nu=0$, $\mu$ to be equal to (our) $2\nu+1$, $a=1$ and $b=\sin r$) we get that
$$
F(-\nu,\nu+1;1;1-\cos^2 r)= F(-\nu,\nu+1;1;\sin^2r)=\int\limits_0^\infty J_{2\nu+1}(x)J_0(x\sin r)dx.
$$
Finally, upon differentiating with respect to $\nu$ and applying the formula 8.486(1), part 6, of \cite{GR07} with $n=1$, the proof is complete.
\end{proof}

The following corollary summarizes different representations of the $\log$-norm of a linear polynomial
which are obtained by combining Proposition \ref{prop_Legendra_formulas} with \eqref{starting identity}.

\begin{cor} Assuming the notation as above and $z\notin \mathcal D\cup\{D\}$, one has that
\begin{equation}\label{starting formula2}
 2\log | P_{D}(z)|=\rho(z)+\log||D||^2 + \frac{d}{d\nu}F(-\nu,\nu+1;1;1-\cos^2 r)\Big|_{\nu=0}.
\end{equation}
Additionally, for any integer $\ell \geq 1$, one has that
\begin{equation} \label{starting identity1}
2\log | P_{D}(z)|=\rho(z)+\log||D||^2 -H_\ell - \sum_{j=1}^{\infty} \frac{(2j+\ell)(-1)^j}{j(j+\ell)} P_j^{(\ell-1,0)}(\cos(2r)).
\end{equation}
Finally, one also has that
\begin{equation}\label{starting formula4}
2\log | P_{D}(z)|=\rho(z)+\log||D||^2+\int\limits_0^\infty \left(\pi Y_1(x)+ \frac{2}{x}J_0(x)\right)J_0(x\sin r)dx.
\end{equation}
\end{cor}

The identities \eqref{starting formula2} and \eqref{starting identity1} will serve as a starting point for the computation
of the equivalent expressions for the  Mahler measure of the linear polynomial $P_D(z)$.  Actually, we will not use
equation \eqref{starting formula4} and present the formula only for possible future interest.

\section{A change of variables formula}

Let $S$ denote the domain of integration in \eqref{Mahler_definition1}, meaning that
$S$ is the subset in the affine chart $\mathcal{Z}_{0} \neq 0$  of $\CC\PP^{n}$ consisting of $n$-tuples
of affine coordinates $(z_1,...,z_n)$ such that
$$
(z_1,...,z_n) =(e^{i\theta_1},\ldots,e^{i\theta_n})
\,\,\,\,\,
\textrm{\rm with}
\,\,\,\,\,
(\theta_1,\ldots,\theta_n) \in [0,2\pi]^n.
$$
We assume that $S$ is equipped with the measure
$$
\mu_S(z)= \frac{1}{(2\pi i)^n}\frac{dz_1}{z_1} \cdots \frac{dz_n}{z_n}
= \frac{1}{(2\pi)^{n}}d\theta_{1}\cdots d\theta_{n}.
$$
The following discussion is based on the material from pages 419-422 of \cite{Wa66} which is
summarized here for the convenience of the reader.

Let $h\in L^{1}([0,1])$, meaning $h(x)$ is an absolutely integrable function for $x \in [0,1]$.
We will view $x$ as a function of $z, D\in \CC\PP^{n}$ by
$$
x =x(z,D) := (\cos(d_{\textrm{FS}}(z,{D})))^{2}=\cos^2 r,
$$
in the notation of the previous section. Consider the integral
\begin{equation*}\label{main ineq_gen}
I(D;h)=\int\limits_{S}h(x(z,D))\mu_S(z).
\end{equation*}
Let us write $D=(\mathcal{W}_0,\ldots,\mathcal{W}_n)=(r_0 e^{i\varphi_0}, \ldots, r_n e^{i\varphi_n})$.  Set
$\mathcal{Z}_{0}=1$ and $\mathcal{Z}_{m} = e^{i\theta_{m}}$ for each integer $m$ from $1$ to $n$ to be a choice of coordinates for $z\in S$.
With this, set
\begin{equation*}\label{dot_prod}
\mathcal{X}:=\sum_{m=0}^{n} \mathcal{Z}_m \overline{{\mathcal{W}}}_m=  \sum_{m=0}^{n} r_m e^{i(\theta_m -\varphi_m)}.
\end{equation*}
Following the discussion on pages 419-422 of \cite{Wa66}, we can view $\mathcal{X}$
as the endpoint of an $(n+1)-$step random walk in two dimensions.  Step number $m$ is of length $r_{m}$, and
the walk occurs in the direction with angle $(\theta_m -\varphi_m)\in [-\pi,\pi]$.  The directions are viewed as independent
and identically distributed random variables, and the probability distribution of each is uniform on the interval $[-\pi,\pi]$.
Let
$$
d(D) = \vert \mathcal{W}_{0} \vert + \cdots + \vert{\mathcal{W}_{n}}\vert
$$
be the $L^{1}$ norm of $D$.  Let $\mathcal{Y}$ be the random variable which is the distance of $\mathcal{X}$ to the
origin.  Observe that for $z\in S$ we can write $x$ as
$$
x=(\cos(d_{\mathrm{FS}}(z,{D})))^{2}=\frac{1}{c(D)^2}\left| \sum_{m=0}^{n} \mathcal{Z}_m \overline{{\mathcal{W}}}_m \right|^2
= \frac{\mathcal{Y}^{2}}{c(D)^{2}}
$$
where $c(D)^{2} = (n+1)\Vert D \Vert^{2}$.
It is proved on page 420 of \cite{Wa66} that for any $u\in[0,d(D)]$ the cumulative
distribution $F_D(u)$ of $\mathcal{Y}$ is given by
$$
\textrm{\rm Prob}(\mathcal{Y} \leq u) = F_D(u)=u\int\limits_{0}^{\infty}J_1(ut)\prod_{m=0}^{n}J_0(r_m t) dt.
$$
Of course, $F_D(u)=0$ for $u<0$ and $F_D(u)=1$ for $u>d(D)$, and $J_0$ and $J_1$ are the classical $J$-Bessel functions.
The probability density function $f_{D}(u)$ of $\mathcal{Y}$ is obtained by differentiating $F_{D}(u)$ with respect to $u$.
Using formula 8.472.1 of \cite{GR07}, we deduce that for $u\in [0,d(D)]$
the function $f_{D}(u)$ is given by
$$
f_D(u)=\int\limits_{0}^{\infty}ut J_0(ut)\prod_{m=0}^{n}J_0(r_m t) dt;
$$
also, $f_{D}(u)$ is equal to zero for $u\notin [0,d(D)]$. When $r_{m}=1$ for all $m$, the
above formula is a classical result of Kluyver \cite{Kl05}; see also formula (2.1) of \cite{BSWZ12}.

With all this,  we can re-write the integral $I(D;h)$ as
\begin{align}\label{main ineq_gen2}\nonumber
I(D;h)&=\int\limits_{S}h(x(z,D))\mu_S(z)
= \int\limits_{0}^{d(D)} h\left(u^{2}/c(D)^{2}\right)f_{D}(u)du
\\&= \int\limits_{0}^{d(D)} h\left(u^{2}/c(D)^{2}\right)\left(\int\limits_{0}^{\infty}ut J_0(ut)\prod_{m=0}^{n}J_0(r_m t) dt\right)du.
\end{align}
Finally, if we let $v=u/c(D)$, we arrive at a general change of variables formula, namely that
\begin{equation}\label{change_of_variables}
\int\limits_{S}h(x(z,D))\mu_S(z)
= c(D)^{2}\int\limits_{0}^{d(D)/c(D)} h(v^{2})\left(\int\limits_{0}^{\infty}vt J_0(c(D)vt)\prod_{m=0}^{n}J_0(r_m t) dt\right)dv.
\end{equation}

Recall that the Cauchy-Schwarz inequality implies that $d(D) \leq c(D)$, so then the above stated assumption
on the function $h$ indeed is sufficient for the above identity to hold.  Indeed, only in the case
when $D$ is a multiple of $(1,\cdots, 1)$ we have $d(D)=c(D)$, otherwise $d(D) < c(D)$.  This well-known
aspect of the Cauchy-Schwarz inequality will be important when we apply \eqref{change_of_variables} to prove
our results.

\section{Proof of the main results}

\subsection{Proof of Theorem \ref{thm:main}}

We begin with equation \eqref{starting formula2} and integrate along $S$ with respect to the measure $\mu_S(z)$. Recall that we derived \eqref{starting formula2} under the condition that $z\notin \mathcal{D}\cup \{ D \}$ where $\mathcal D$ is the divisor of the polynomial $P_D$ on $\CC\PP^n$
The left-hand-side of the resulting formula is $2m(P_{D})$, so it remains to compute the integral of the right-hand side.
Recall that $r=d_{\mathrm{FS}}(z,{D})$, in the notation of Section \ref{sec:log norm} and $x=x(z,D)=\cos ^2 r$, in the notation of the previous section.

 When $(\mathcal D\cup \{  D\}) \cap S = \emptyset$, there is an $\epsilon > 0$ such that for all $z \in S$ we have the bound $d_{\mathrm{FS}}(z,{D}) \geq \epsilon > 0$.
Hence, the hypergeometric function $F(-\nu,\nu+1;1;1-\cos^2 r)$ converges absolutely and uniformly in $\nu$ since $f$ is positive and bounded away from $0$.
Therefore, we may differentiate the series expansion for the hypergeometric function $F(-\nu,\nu+1;1;1-\cos^2 r)$ term by term to get that
\begin{equation}\label{hyper_series}
\frac{d}{d\nu}F(-\nu,\nu+1;1;1-\cos^2 r)\Big|_{\nu=0}= -\sum_{j=1}^{\infty}\frac{1}{j}(1-\cos^2 r)^j.
\end{equation}
Each term in the series \eqref{hyper_series} is non-negative, so the monotone convergence theorem applies to
give that
\begin{equation}\label{mahler_hyper_formula}
2m(P_{D})=2 \log c(D) - \int\limits_S \left(\sum_{j=1}^{\infty}\frac{1}{j}(1-\cos^2 r)^j \right)\mu_S(z).
\end{equation}
Suppose that $(\mathcal D\cup \{ D \}) \cap S \neq \emptyset$.  Choose an $\epsilon > 0$ and set
$$
S_{\epsilon}:= \{z \in S : d_{\mathrm{FS}}(z,w) \geq \epsilon \,\,\textrm{\rm for all} \,\, w \in  \mathcal D \cup \{D \}\}.
$$
By preceding as above, we get arrive at the formula that
\begin{equation}\label{mahler_hyper_formula_epsilon}
2 m(P_{D};S_{\epsilon}) =2 \frac{\textrm{\rm vol}(S_{\varepsilon})}{(2\pi)^{n}}\log c(D)
 - \int\limits_{S_{\epsilon}} \left(\sum_{j=1}^{\infty}\frac{1}{j}(1-\cos^2 r)^j \right)\mu_S(z).
\end{equation}
where
$$
m(P_{D};S_{\epsilon}) := \frac{1}{(2\pi)^{n}}\int\limits_{S_{\epsilon}}
\log  \left| P_D(1,e^{{i\theta _{1}}},e^{{i\theta _{2}}},\ldots ,e^{{i\theta _{n}}})\right| d\theta _{1}\,d\theta _{2}\cdots d\theta _{n}
$$
If $(\mathcal D \cup \{ D \})\cap S \neq \emptyset$, then $(\mathcal D \cup \{ D \}) \cap S$ has $\mu_S$ measure zero, as also noted on page 270 of \cite{De97}.  Hence,
the function $\log\vert P_D\vert$ lies in $L^{1}(\mu_S)$.  Therefore, by letting
$\epsilon$ approach zero, we have, by the monotone convergence theorem, that \eqref{mahler_hyper_formula_epsilon}
becomes \eqref{mahler_hyper_formula}.
In other words, in both cases when $(\mathcal D \cup \{ D \})\cap S = \emptyset$ and when $(\mathcal D\cup\{ D \})\cap S \neq \emptyset$, we arrive at
\eqref{mahler_hyper_formula}, which we now study.

The series on the right-hand side of \eqref{mahler_hyper_formula} is a series of non-negative functions; hence, we can interchange
the sum and the integral to get, for any integer $N\geq 1$
\begin{equation}\label{hypergeom integral}
2m(P_{D}) -2 \log c(D) +\sum_{j=1}^{N}\frac{1}{j} \int\limits_S (1-\cos^2 r)^j\mu_S(z) =-\sum_{j=N+1}^{\infty}\frac{1}{j} \int\limits_S (1-\cos^2 r)^j\mu_S(z).
\end{equation}
In order to complete the proof of Theorem \ref{thm:main} we will evaluate the integral over $S$ and show that the right-hand side of \eqref{hypergeom integral}
is dominated by $2\Gamma(3/4)G(n,D)/(3N^{3/4})$ for a certain constant $G(n,D)$ which depends solely on $n$ and $D$.
We will now do so, and the formula for $G(n,D)$ is given in \eqref{G(n,D) defin} below.

As before, $D$ is identified with $(w_1,...,w_n)=(\mathcal{W}_0,...,\mathcal{W}_n)\in\CC\PP^n$ (recall that the affine chart is chosen so that $\mathcal{W}_0\neq 0$). For $z \in S$, we have that
$$
x^k(z,D)=(\cos r)^{2k}= \frac{\left|1+ \sum_{\ell=1}^n \overline{w}_\ell e^{i\theta_\ell} \right|^{2k}}{(1+n)^k \left(1+ \sum_{\ell=1}^n |w_\ell|^2\right)^k },
$$
so then
\begin{equation}\label{integral of f to k}
\frac{1}{(2\pi)^n} \int\limits_{0}^{2\pi}\cdots \int\limits_{0}^{2\pi} x^k(z,D)d\theta_1 \cdots d\theta_n= \frac{a_1(n,k,D)}{(1+n)^k \left(1+ \sum_{\ell=1}^n |w_\ell|^2\right)^k} ,
\end{equation}
where $a_1(n,k,D)$ denotes the constant term in the expression
$$
x^k(z,D)=
\left|1+ \sum_{\ell=1}^n \overline{w}_\ell e^{i\theta_\ell} \right|^{2k} =\left(1+ \sum_{\ell=1}^n
w_\ell e^{-i\theta_\ell}\right)^k \left(1+ \sum_{\ell=1}^n \overline{w}_\ell e^{i\theta_\ell}\right)^k.$$
The multinomial theorem implies that
\begin{align*}
a_1(n,k,D) &= \sum_{\ell_0+\ldots+\ell_n=k, \ell_m\geq 0, m=1,\ldots n} \binom{k}{\ell_0,\ell_1,\ldots,\ell_n}^2|w_1|^{2\ell_1}\cdots|w_n|^{2\ell_n}
= \frac{a(n,k,D)}{|\mathcal{W}_0|^{2k}}.
\end{align*}
Therefore,
$$
\int\limits_S(1-x(z,D))^j\mu_S(z)= \sum_{k=0}^{j} \binom{j}{k}(-1)^k\frac{a(n,k,D)}{c(D)^{2k}}.
$$
Inserting this into \eqref{hypergeom integral} we get
\begin{equation}\label{pre-final f-la}
\left|2m(P_{D}) -2 \log c(D) +\sum_{j=1}^{N}\frac{1}{j} \sum_{k=0}^{j} \binom{j}{k}\frac{(-1)^ka(n,k,D)}{c(D)^{2k}} \right| \leq \sum_{j=N+1}^{\infty}\frac{1}{j} \int\limits_S (1-x(z,D))^j\mu_S(z).
\end{equation}
To complete the proof of the theorem, it remains to deduce a uniform bound for the series in \eqref{pre-final f-la}.

Let us apply the change of variables formula \eqref{change_of_variables} and write, for any $j \geq 1$,
\begin{equation}\label{hyper_tail}
\int\limits_S(1-x(z,D))^j\mu_S(z) =
c(D)^{2}
\int\limits_{0}^{d(D)/c(D)} (1-v^2)^{j}\left(\int\limits_{0}^{\infty}vt J_0(c(D)vt)\prod_{m=0}^{n}J_0(r_m t) dt\right)dv.
\end{equation}
Since $d(D)/c(D)\leq 1$, we have that $v\leq 1$, hence $\max \{1,\frac{\pi}{2}c(D)vt\}\geq v \max \{1,\frac{\pi}{2}c(D)t\}$. By using the bound \eqref{J0 basic ineq}, we get that
$$
|J_0(c(D)vt)| \leq v^{-1/2} \max \{1,\frac{\pi}{2}c(D)t\}^{-1/2}.
$$
Therefore,
\begin{align*}
\int\limits_{0}^{d(D)/c(D)} (1-v^2)^{j}v |J_0(c(D)vt)|dv
&\leq \left(\int\limits_{0}^{1}(1-v^2)^{j}v^{1/2}dv\right)\max\{1,\frac{\pi}{2} c(D)t\}^{-1/2}\\&= \frac{1}{2}\left(\int\limits_{0}^{1}(1-u)^{j}u^{-1/4}du\right)\max\{1,\frac{\pi}{2} c(D)t\}^{-1/2}
\\&= \frac{\Gamma(j+1)\Gamma(3/4)}{2\Gamma(j+1+3/4)}\max \{1,\frac{\pi}{2}c(D)t\}^{-1/2}\\&\leq \frac{\Gamma(3/4)}{2j^{3/4}} \max \{1,\frac{\pi}{2}c(D)t\}^{-1/2},
\end{align*}
where we have applied \cite{GR07}, formula 3.196.3 with $a=0$, $b=1$, $\mu=3/4$ and $\nu=j+1$ in order to evaluate the integral with respect to $u$.
Trivially, $\sum_{j=N+1}^{\infty} j^{-7/4} \leq \frac{4}{3} N^{-3/4}$. Hence, after multiplying by $1/j$ and summing over $j\geq N+1$ in \eqref{hyper_tail},
we arrive at the bound
$$
\sum_{j=N+1}^{\infty}\frac{1}{j}\int\limits_S(1-x(z,D))^j\mu_S(z)
\leq \frac{2\Gamma(3/4)}{3} \frac{G(n,D)}{N^{3/4}}
$$
where
\begin{equation}\label{G(n,D) defin}
G(n,D)= c(D)^{2} \int\limits_{0}^{\infty} t\left(\max \{1,\frac{\pi}{2}c(D)t\}\right)^{-\tfrac{1}{2}}
\prod_{m=0}^{n} |J_0(r_m t) | dt.
\end{equation}
By combining with \eqref{pre-final f-la} we obtain the inequality
$$
\left|2m(P_D)-2E_1(N;n,D)\right|\leq \frac{2\Gamma(3/4)}{3} \frac{G(n,D)}{N^{3/4}},
$$
where $E_1(N;n,D)$ is defined by \eqref{defn E_1}.  We have now completed the
proof of the first inequality in \eqref{upper bounds}.  When dividing by $2$ and
letting $N \rightarrow \infty$, we also have completed the proof of Theorem \ref{thm:main}.\qed

As a concluding comment, let us point out a further refinement which will yield an
elementary bound for $G(n,D)$.
By using \eqref{J0 basic ineq}, we arrive at the inequality
\begin{equation}\label{G_bound_Bessel}
G(n,D)\leq c(D)^{2} \int\limits_{0}^{\infty} t\left(\max \{1,\frac{\pi}{2}c(D)t\}\right)^{-\tfrac{1}{2}}
\prod_{m=0}^{n} \left(\max \{1,\frac{\pi}{2}r_{m}t\}\right)^{-\tfrac{1}{2}}dt.
\end{equation}
Without loss of generality, assume that $r_0\leq\ldots\leq r_n$, so then $r_{n}\leq c(D)$.
If $n \geq 3$, the integral is convergent since the integral is $O(t^{-n/2})$ for $t > 2/(\pi r_{0})$.
In order to evaluate \eqref{G_bound_Bessel}, we can write the domain of integration as
$$
\int\limits_0^{\infty}= \int\limits_0^{2/(\pi c(D))} + \int\limits_{2/(\pi c(D))}^{2/(\pi r_n)} +
\cdots + \int\limits^{2/(\pi r_j)}_{2/(\pi r_{j+1})}+ \cdots + \int\limits^{\infty}_{2/(\pi r_{0})}.
$$
All integrals are elementary, so then the evaluation of each integral will yield an explicit
and elementary upper bound for $G(n,D)$.  In the case when
not all $r_{j}$'s are distinct, some of these integrals are zero, and the exponents for each integrand
need to take into account the multiplicity of $r_{j}$ in set of components of $\mathcal D$.  
The
computations are elementary.

\subsection{Proof of Theorem \ref{thm:main2} when $\ell =1$}

Choose any integer $N \geq 1$.  when $\ell=1$ we can write \eqref{starting identity1} for $z\notin (\mathcal D\cup \{ D \})$ as
\begin{align}\label{starting formula3with N}\nonumber
2\log | P_{D}(z)|=\rho(z)+&\log||D||^2 - 1-\sum_{j=1}^{N}(-1)^j\frac{2j+1}{j(j+1)}P_j(\cos(2r))
\\& -\sum_{j=N+1}^{\infty}(-1)^j\frac{2j+1}{j(j+1)}P_j(\cos(2 r)),
\end{align}
where, as before $r=d_{\mathrm{FS}}(z,{D})$.  We now will utilize the following three points:
The identity $\cos(2r)=2\cos^2 r -1$; formula 8.962.1 from \cite{GR07} for the Jacobi polynomial with $\alpha=\beta=0$; and that hypergeometric function $F(j+\ell,-j;1;\cos^2 r)$ at these values is the finite sum.  In doing so, we arrive at the formula
$$
P_j^{(\ell-1,0)}(2\cos^2 r-1)=(-1)^j\sum_{k=0}^j \binom{j+\ell+k-1}{k}\binom{j}{k}(-1)^k(\cos r)^{2k},
$$
which holds for all $z\in S$ when $(\mathcal D \cup\{ D \})\cap S = \emptyset$.  For now, we will assume $(\mathcal D \cup \{ D \})\cap S = \emptyset$.
By integrating this equation with respect to $S$ and employing \eqref{integral of f to k} we get
\begin{equation}\label{integral of P l-1}
(-1)^j\int\limits_S P_j^{(\ell-1,0)}(\cos(2r))\mu_S(z)=  \sum_{k=0}^j \binom{j+\ell+k-1}{k}\binom{j}{k}(-1)^k \frac{a(n,k,D)}{c(D)^{2k}}.
\end{equation}
By integrating \eqref{starting formula3with N} with respect to $S$ and applying \eqref{integral of P l-1} with $\ell=1$ we arrive at
\begin{equation}\label{First bound}
\left|m(P_D)-E_2(N;n,D)\right| \leq \frac{1}{2} \sum_{j=N+1}^{\infty}\frac{2j+1}{j(j+1)}\int\limits_{S}\left|P_j(\cos(2d_{\mathrm{FS}}(z, {D})))\right|\mu_S(z),
\end{equation}
where $E_2(N;n,D)$ is defined by \eqref{defn E_2}.

Before studying the right-hand-side of \eqref{First bound}, let us address the setting when $(\mathcal D \cup \{ D \})\cap S \neq \emptyset$. Indeed, the extension of \eqref{First bound} to the case when $(\mathcal D \cup \{ D \})\cap S \neq \emptyset$ follows the method of proof of \eqref{mahler_hyper_formula} in that case.  By integrating over $S_{\epsilon}$ rather than all of $S$, one gets an analogue of \eqref{First bound} for $\epsilon >0$.  At that point, one lets $\epsilon$ approach zero.  The function $\log \vert P_{D}\vert$
and each Jacobi polynomial $P_{j}^{(\ell-1,0)}$ is in $L^{1}(S)$, so one obtains the left-hand-side of \eqref{First bound} as $\epsilon$ approaches zero.
As for the right-hand-side of \eqref{First bound}, one uses the monotone convergence theorem, as in the proof of
\eqref{hypergeom integral}.  With this, one shows that \eqref{First bound} also holds when $(\mathcal D\cup\{ D \})\cap S \neq \emptyset$.

It is left to study the series on the right-hand side of \eqref{First bound}.  We will use inequality \eqref{poly_bound1}
with $\cos(2r)=\cos(2d_{\mathrm{FS}}(z, {D}))$ instead of $x$.
Recall the notation $x(z,D) = \cos^2 (d_{\mathrm{FS}}(z, {D}))$, so then the
inequality \eqref{poly_bound1} gives that
\begin{equation}\label{P_formula}
\int\limits_{S}\left|P_j(\cos(2d_{\mathrm{FS}}(z, {D})))\right|\mu_S(z)
\leq \frac{2}{\sqrt{\pi}\sqrt[4]{2}}\frac{1}{\sqrt{2j+1}}\int\limits_{S}\left(x(z,D)(1-x(z,D))\right)^{-1/4}\mu_S(z).
\end{equation}
Therefore, the right-hand-side of \eqref{First bound} can be bounded from above by
$$
\frac{1}{\sqrt{\pi}\sqrt[4]{2}}\left(\sum_{j=N+1}^{\infty}\frac{\sqrt{2j+1}}{j(j+1)}\right)
\int\limits_{S}\left(x(z,D)(1-x(z,D))\right)^{-1/4}\mu_S(z).
$$
The goal is to make all bounds effective and explicit, so we shall.  Trivially, we have that
$$
\sum_{j=N+1}^{\infty}\frac{\sqrt{2j+1}}{j(j+1)}
\leq \sum_{j=N+1}^{\infty}\frac{2j^{-1/2}}{j^{2}} \leq \frac{4}{\sqrt{N}},
$$
which, when combined with \eqref{First bound}, yields the inequality
\begin{equation}\label{Second_bound}
\left|m(P_D)-E_2(N;n,D) \right| \leq
\frac{C}{\sqrt{N}}\int\limits_{S}\left(x(z,D)(1-x(z,D))\right)^{-1/4}\mu_S(z)
\,\,\,
\textrm{\rm for}
\,\,\,
C = \frac{4}{\sqrt{\pi}\sqrt[4]{2}}.
\end{equation}
The integral in \eqref{Second_bound} can be re-written using the change of variables formula
\eqref{change_of_variables}.  In doing so, it becomes
\begin{equation}\label{Upper_bound_integral}
H(n,D) = c(D)^{2}\int\limits_{0}^{d(D)/c(D)} v^{-1/2}(1-v^{2})^{-1/4}
\left(\int\limits_{0}^{\infty}vt  J_0(c(D)vt) \prod_{m=0}^{n} J_0(r_m t)  dt\right)dv.
\end{equation}
The Fubini-Tonelli theorem then implies that
\begin{equation}\label{I(D) bound}
H(n,D)\leq c(D)^{2} \int\limits_{0}^{\infty} t \prod_{m=0}^{n}\vert J_0(r_m t)\vert\cdot\left( \int\limits_{0}^{1} v^{\tfrac{1}{2}}\left(1-v^2\right)^{-\tfrac{1}{4}}
\vert J_0(c(D)tv)\vert dv\right) dt.
\end{equation}
The Cauchy-Schwarz inequality together with the elementary inequality \eqref{J0 basic ineq} for the $J$-Bessel function gives the inequality
$$
\int\limits_{0}^{1} v^{\tfrac{1}{2}}\left(1-v^2\right)^{-\tfrac{1}{4}}\vert J_0(c(D)tv)\vert dv
\leq \sqrt{\frac{\pi}{2}}\left(\int\limits_{0}^{1} v J_0^2(c(D)tu)dv\right)^{\tfrac{1}{2}}\leq
\sqrt{\frac{\pi}{2}}\left(\max \{1,\frac{\pi}{2}c(D)t\}\right)^{-\tfrac{1}{2}}.
$$
Finally, by substituting this inequality into \eqref{I(D) bound}, we arrive at the bound
\begin{equation*}\label{I(D) bound new}
H(n,D)\leq c(D)^{2} \sqrt{\frac{\pi}{2}}\int\limits_{0}^{\infty} t\left(\max \{1,\frac{\pi}{2}c(D)t\}\right)^{-\tfrac{1}{2}}
\prod_{m=0}^{n} |J_0(r_m t)| dt= \sqrt{\frac{\pi}{2}} G(n,D).
\end{equation*}
Therefore,
$$
\left|m(P_D)-E_2(N;n,D) \right| \leq \frac{2 \sqrt[4]{2}}{\sqrt{N}}G(n,D),
$$
which proves the second inequality in \eqref{upper bounds}.  Formula \eqref{measure formula 3} follows by letting $N$
tend to infinity. \qed

\subsection{Proof of equation \eqref{measure formula 4}}

Equation \eqref{measure formula 4} follows by a direct manipulation of the inner sum appearing in \eqref{measure formula 3}.
To ease the notation, we will set $b(n,k,D)=(-1)^k a(n,k,D)/c(D)^{2k}$.  Then, we can write
\begin{align*}
\sum_{j=1}^{\infty}\frac{2j+1}{j(j+1)}&\sum_{k=0}^j\binom{j+k}{k}\binom{j}{k} b(n,k,D)= 1-\frac{2}{n+1}
\\&+\sum_{j=2}^{\infty}\frac{1}{j}\left(\sum_{k=0}^j\binom{j+k}{k}\binom{j}{k} b(n,k,D)+\sum_{k=0}^{j-1}\binom{j+k-1}{k}\binom{j-1}{k} b(n,k,D)\right).
\end{align*}
It is elementary to prove that
\begin{align*}
\sum_{k=0}^j\binom{j+k}{k}\binom{j}{k} b(n,k,D)&+\sum_{k=0}^{j-1}\binom{j+k-1}{k}\binom{j-1}{k} b(n,k,D)\\&=
2\sum_{k=0}^j\binom{j+k-1}{k}\binom{j}{k} b(n,k,D),
\end{align*}
which completes the proof of equation \eqref{measure formula 4}. \qed

\subsection{Proof of Theorem \ref{thm:main2} when $\ell\geq 2$}

Assume $\ell \geq 2$ and $D\neq r(1,1,\ldots,1)$ for some $r\neq 0$. Proceeding as above, we integrate
\eqref{starting identity1} along $S$ with respect to $\mu_S(z)$ and employ \eqref{integral of P l-1} to arrive at
\begin{multline}\label{measure for ell >1}
\left|m(P_D)-\log c(D)-\frac{H_\ell}{2}-\frac{1}{2}\sum_{j=1}^{N}
\frac{2j+\ell}{j(j+\ell)}\sum_{k=0}^j\binom{j+\ell+k-1}{k}\binom{j}{k} \frac{(-1)^k a(n,k,D)}{c(D)^{2k}} \right| \\
\leq \frac{1}{2}\sum_{j=N+1}^{\infty}
\frac{2j+\ell}{j(j+\ell)}\int\limits_S \left|P_j^{(\ell-1,0)}(\cos(2d_{\mathrm{FS}}(z, {D})))\right|\mu_S(z).
\end{multline}
In the notation as above, set
\begin{equation}\label{limit_r_s2}
H(j;\ell,D):= \int\limits_{S}\left|P_j^{(\ell-1,0)}(\cos(2d_{\mathrm{FS}}(z, {D})))\right|\mu_S(z)=  \int\limits_{S}\left|P_j^{(\ell-1,0)}(2 x(z,D)-1)\right|\mu_S(z).
\end{equation}
Equation \eqref{change_of_variables} applies, so then
$$
H(j;\ell,D) = c(D)^{2}\int\limits_{0}^{d(D)/c(D)}\left| P^{(\ell-1,0)}_{N}(2v^{2}-1)\right|\left(\int\limits_{0}^{\infty}vt J_0(c(D)vt)\prod_{m=0}^{n}J_0(r_m t) dt\right)dv.
$$
We now apply the bound in \eqref{poly_bound2} with $x=2v^2-1 \in [-1,2(d(D)/c(D))^2-1]\subset[-1,1)$ to get
$$
(1-v^2)^{1/4}v^{1/2}\left(1-v^2\right)^{(\ell-1)/2}|P^{(\ell-1,0)}_N(2v^2-1)|\leq \frac{6\sqrt{2}}{\sqrt{2j+\ell}}.
$$
Recall that $d(D)$ is the $L^{1}$ norm of $D$, and $c(D)^{2} = (n+1)\Vert D \Vert^{2}$.  By the Cauchy-Schwarz inequality,
$d(D)/c(D) \leq 1$ with equality if and only if $D = r(1,1,\cdots, 1)$ for some non-zero $r$. Since
we assume that $D \neq r(1,1,\cdots, 1)$, it follows that $d(D)/c(D)<1$, hence
$$
|P^{(\ell-1,0)}_N(2v^2-1)| \leq \frac{ A(D,\ell)}{ \sqrt{2j+\ell}} v^{-1/2}(1-v^{2})^{-1/4},
$$
where
$$
A(D,\ell) = 6\sqrt{2} \left(1-\frac{d(D)^2}{c(D)^2}\right)^{-(\ell-1)/2}.
$$
Therefore
$$
H(j;\ell,D)\leq\frac{ c(D)^{2}A(D,\ell)}{\sqrt{2j+\ell}}\int\limits_{0}^{\infty} t \prod_{m=0}^{n}\vert J_0(r_m t)\vert\cdot\left( \int\limits_{0}^{1} v^{\tfrac{1}{2}}\left(1-v^2\right)^{-\tfrac{1}{4}}
\vert J_0(c(D)tv)\vert dv\right) dt
$$
The integral on the right-hand side of the above equation is the same as the integral which appears in \eqref{I(D) bound}.
So then the argument following \eqref{I(D) bound} applies and gives the inequality
$$
H(j;\ell,D)\leq\sqrt{\frac{\pi}{2}}\frac{A(D,\ell)}{\sqrt{2j+\ell}} G(n,D).
$$
Consequently, we have shown that
$$
\sum_{j=N+1}^{\infty}
\frac{2j+\ell}{j(j+\ell)}\int\limits_S \left|P_j^{(\ell-1,0)}(\cos(2\textrm{\rm dist}_{\mathrm{FS}}(z, {D})))\right|\mu_S(z)\leq \sqrt{\frac{\pi}{2}}A(D,\ell)G(n,D)\sum_{j=N+1}^{\infty}
\frac{ \sqrt{2j+\ell}}{j(j+\ell)}.
$$
This proves that the right-hand side of \eqref{measure for ell >1} is bounded by $\sqrt{\pi}A(D,\ell)G(n,D)N^{-1/2}$.

With this, the proof of Theorem \ref{thm:main2} when $\ell\geq 2$ follows upon letting $N\to\infty$ in \eqref{measure for ell >1}. \qed

\section{Concluding remarks}

\subsection{Proof of Corollary \ref{cor. comparison}}
Let us rewrite a result from \cite{R-VTV04} in our notation. Specifically, equation (4.7) from \cite{R-VTV04} becomes the formula that
\begin{equation}\label{comp. fla 2}
2m(P_D)= \log||D||^2-\gamma-\sum_{m=1}^{\infty}\frac{1}{m}\sum_{\ell=0}^m \binom{m}{\ell}\frac{(-1)^{\ell} a(n,\ell,D)}{l!||D||^{2l}}.
\end{equation}
Trivially,
$$
\sum_{\ell=0}^1 \binom{1}{\ell}\frac{(-1)^{\ell} a(n,\ell,D)}{l!||D||^{2l}}=0,
$$
from which \eqref{comparison equation} follows immediately by comparing \eqref{main expression} with \eqref{comp. fla 2}. \qed

\subsection{Additional formulas for Mahler measures}

For any $n\geq 3$ choose any $D$ and one of the formulas we have proved, say Theorem \ref{thm:main}.  For any
integer $\tilde{n}>n$, let $\widetilde{D}$ be the vector of coefficients whose first $n$ components is $D$ and whose
last $\tilde{n}-n$ coordinates are zero.  The normalization in \eqref{Mahler_definition} is such that
$m(P_{D}) = m(P_{\widetilde{D}})$.  Also, for any $k$ we have that $a(\tilde{n},k,\widetilde{D}) = a(n,k,D)$.  However,
$$
c(\widetilde{D})^2 = (\tilde{n}+1)\Vert \widetilde{D}\Vert^{2} =
(\tilde{n}+1)\Vert D\Vert^{2} = \frac{\tilde{n}+1}{n+1}c(D)^2.
$$
Let us set $m=\tilde{n}-n$.  With this, the main formula in Theorem \ref{thm:main} becomes the
statement that for any $m \geq 0$, one has that
\begin{align} \label{comb identities}
m(P_D)= \log c(D) + \frac{1}{2}\log\left(\frac{n+m+1}{n+1}\right)- \frac{1}{2} \sum_{j=1}^{\infty}\frac{1}{j}\sum_{k=0}^j\binom{j}{k} \frac{(-1)^k a(n,k,D)(n+1)^{2k}}{c(D)^{2k}(n+m+1)^{2k}}.
\end{align}
Similar identities can be proved by the by the same means from Theorem \ref{thm:main2}.

Equation \eqref{comb identities} with $m=0$ and $m\geq 1$ yields the following curious combinatorial identity, similar to \eqref{comparison equation}
\begin{align} \label{comb identities2}
\log\left(\frac{n+m+1}{n+1}\right)= \sum_{j=1}^{\infty}\frac{1}{j}\sum_{k=0}^j\binom{j}{k} \frac{(-1)^k a(n,k,D)}{c(D)^{2k}}\left[\left(\frac{n+1}{n+m+1}\right)^{2k}-1\right],
\end{align}
which holds true for any $m\geq 1$.

Note that all estimates we have derived for $G(n,D)$ grow for fixed $D$ as $n$ increase.  As such, the
above considerations do not seem to aid with convergence issues when applying our results to numerical estimations.

\subsection{The excluded cases when $n=2$ and $D=r(1,\cdots, 1)$}

In the case $n=2$, it may be possible to revisit our computations and obtain bounds.  For example, rather than using
\eqref{J0 basic ineq}, one could use the asymptotic formula
$$
J_{0}(x) = \sqrt{\frac{2}{\pi x}}\cos(z-\pi/4) + O(x^{-3/2})
\,\,\,\,\,
\textrm{\rm as $x \rightarrow \infty$.}
$$
The oscillatory bounds may be such that one an derive a finite bound for $G(2,D)$
and possibly improved bounds for $G(n,D)$ for $n \geq 3$ as well.  These considerations are
undertaken in \cite{AMS20}.

The exclusion of the point $D=r(1,\cdots, 1)$ in Theorem \ref{thm:main2} comes from the problem
of deriving bounds for the $P^{(\ell-1,0)}_{j}(r)$ near $r=1$; see \eqref{measure for ell >1}.
For this, one can seek to employ Hilb-type formulas; see page 197 of \cite{Sz74}, page 6 of
\cite{BZ07} or page 980 of \cite{FW85}.  We will leave these questions for
future consideration.

\subsection{Other choices for $S$}

We shall now describe how the approach taken in this paper can be generalized.  In doing so, we
will be somewhat vague in our discussion.

For this section let $S$ be a ``nice'' set in $\CC\PP^n$ with a ``nice'' measure $\mu_{S}$. One example is the product of
circles in an affine chart of $\CC\PP^n$ and $\mu_{S}$ is the translation invariant metric on each circle.
Let us define
$$
m_{S}(F_{\mathcal{D}}) = \int\limits_{S}\log \Vert F_{\mathcal{D}}\Vert^{2}_{\mu}(z) \mu_{S}(z),
$$
where $F_{\mathcal{D}}$ is a holomorphic form on $\CC\PP^n$ with divisor $\mathcal{D}$. The invariant
$m_{\CC\PP^n}(F_{\mathcal{D}})$ is obtained by integrating with respect to the Fubini-Study metric.
By using the spectral expansion of the Green's function and the proof of Proposition \ref{prop:Thm4 improved}
we obtain a general formula, namely
\begin{equation}\label{M_measure_general}
4\pi \textrm{\rm vol}_{\mu}(\CC\PP^n)\sum\limits_{\lambda_{j}> 0}\frac{1}{\lambda_{j}}
\left(\int\limits_{S}\psi_{j} \mu_{S}\right)\left(\int\limits_{\mathcal{D}}\overline{\psi}_{j} \mu_{\mathcal{D}}
\right) = \textrm{\rm vol}_{\mu}(\CC\PP^n) m_{S}(F_{\mathcal{D}}) - \textrm{\rm vol}_\mu(S) m_{\CC\PP^n}(F_{\mathcal{D}}).
\end{equation}
In the above calculations, we were able to express the series
$$
\sum\limits_{\lambda_{j}> 0}\frac{1}{\lambda_{j}}
\left(\int\limits_{S}\psi_{j} \mu_{S}\right)\left(\int\limits_{\mathcal D}\overline{\psi}_{j} \mu_{\mathcal D}
\right)
$$
as a sum of integrals of Jacobi polynomials by using the Radon transform.  When this method
applies, one then obtains a series of Legendre or Jacobi polynomials at various arguments which
are then integrated over $S$.  It certainly seems plausible that our approach will apply in
other settings.

In particular, let us note that in \cite{LM18} the authors considered the case when $S$ is
a product of circles with different radii.  It seems as if the methodology developed in this article
will apply to give analogues of  Theorem \ref{thm:main}, Theorem \ref{thm:main2} and Theorem \ref{thm:series}
in that setting.

\subsection{Estimates for canonical heights} \label{subsect: canonical heights}

The contents of this section are based on comments from an anonymous referee; we gratefully acknowledge
them for sharing their mathematical insight.

First, let us rephrase our study in the context of Arakelov theory.
Following the work in \cite{Ma00}, the calculation of Mahler measures is manifest within the study
of arithmetic intersection theory.  As such, one is naturally led to determine suitable Green's currents associated
to a divisor $\mathcal D$.  There are two immediate possibilities.  First, if $P_{D}$ is a polynomial whose divisor
is $\mathcal D$, then the function
$$
g_{\mathcal D}(z) = - \log (\Vert P_{D}(z)\Vert^{2}_{\textrm{\rm FS}})
$$
is one such Green's function.  In this expression, we have used $z = (\mathcal{Z}_0,\ldots,\mathcal{Z}_n)$
and the subscript ``FS'' to denote the Fubini-Study metric. Second, from above, one also can consider the function
$$
\tilde{g}_{\mathcal D}(z)=
\int\limits_{\mathcal D}G_{\CC\PP^n}(z,w)\mu_{\mathcal D}(w).
$$
The difference $g_{\mathcal D}(z) - \tilde{g}_{\mathcal D}(z)$ admits a smooth extension across the divisor
$\mathcal D$, from which one can prove that
$$
\textrm{\rm d}_{z}\textrm{\rm d}_{z}^{c}\left(g_{\mathcal D}(z) - \tilde{g}_{\mathcal D}(z)\right) = 0.
$$
Therefore, there is a constant $B(D,n)$, which depends on the coefficients $D$ of $P_{D}$ and
the dimension $n$, such that
$$
g_{\mathcal D}(z) - \tilde{g}_{\mathcal D}(z) = B(D,n).
$$
At this point, we have arrived at the second displayed line in the proof of Proposition \ref{prop:Thm4 improved}.
Our subsequent analysis addresses the details of establishing normalizations and evaluation of
the constant $B(D,n)$, as well as the study of $\tilde{g}_{\mathcal D}(z)$ which we undertake via analytic
continuation through the generalization of Kronecker limit formula from \cite{CJS20}.

Second, we now have an opportunity to restate the contents of Theorem \ref{thm:series} as follows.
The bounds in \eqref{upper bounds} provide estimates for the Mahler measure in terms of either
$E_{1}(N;n,D)$ or $E_{2}(N;n,D)$.  Furthermore, the constant $G(n,D)$ can be explicitly computed when,
for example, one combines \eqref{G(n,D) defin} and \eqref{J0 basic ineq}.  As such, Theorem
\ref{thm:series} provides a means by which one can effectively and efficiently estimate the canonical
height which was computed on page 107 of \cite{Ma00}.

\subsection{Reinterpreting Mahler measures}

The results in the present paper follow from the Kronecker-type limit formula derived
in \cite{CJS20}.  The setting of \cite{CJS20} was that of a general K\"ahler manifold $X$
and $\mathcal D$ is a divisor which is smooth up to codimension two.  In this article we took $X$
to be $\CC\PP^{n}$ and $\mathcal D$ to be a hyperplane.  From this initial point, we then delved into
detailed computations and identities involving the Legendre polynomials, Jacobi polynomials
and $J$-Bessel functions.  However, the foray into special function theory was expected.
After all, in many instances one knows that heat kernels can be expressed in terms of
spherical functions, and the Green's function can be computed from the heat kernel; see page 436
of \cite{CJS20} and references therein.  Additionally, Jacobi polynomials and Jacobi functions
are known to be present in such aspects of harmonic analysis; see, for example, \cite{Ko84}.

We find it quite interesting that \eqref{M_measure_general} can be viewed
in the setting of harmonic analysis, and possibly beyond.

\medskip

\noindent
James W. Cogdell \\
Department of Mathematics \\
Ohio State University \\
231 W. 18th Ave. \\
Columbus, OH 43210 \\
U.S.A. \\
e-mail: cogdell@math.ohio-state.edu

\vspace{5mm}\noindent
Jay Jorgenson \\
Department of Mathematics \\
The City College of New York \\
Convent Avenue at 138th Street \\
New York, NY 10031\\
U.S.A. \\
e-mail: jjorgenson@mindspring.com

\vspace{5mm}\noindent
Lejla Smajlovi\'c \\
Department of Mathematics \\
University of Sarajevo\\
Zmaja od Bosne 35, 71 000 Sarajevo\\
Bosnia and Herzegovina\\
e-mail: lejlas@pmf.unsa.ba

\end{document}